\documentclass[12pt,reqno]{amsart}
\usepackage[utf8]{inputenc}
\usepackage[english]{babel} 
\usepackage{amsmath} 
\usepackage{amssymb}
\usepackage{amsthm}
\usepackage{amscd,amsfonts,epsfig,graphics}
\usepackage{bm} 
\usepackage{eucal} 
\usepackage{verbatim}  
\usepackage{hyperref} 
\usepackage{enumitem} 
\usepackage[left=1in,top=1in,right=1in,bottom=1in]{geometry} 
\usepackage[usenames,dvipsnames,svgnames,table]{xcolor} 
\usepackage{parskip} 
\usepackage{microtype} 
\usepackage{csquotes}
\usepackage{lmodern}
\usepackage[T1]{fontenc}
\usepackage{tikz}
\newcommand*\circled[1]{\tikz[baseline=(char.base)]{
   \node[shape=circle,draw,inner sep=1pt] (char) {#1};}}

\title{Factorizations of $k$-Nonnegative Matrices}

\author{Sunita~Chepuri}
\address{Department of Mathematics, University of Minnesota, Minneapolis, MN 55455}
\email{chepu003@umn.edu}
\author{Neeraja~Kulkarni}
\address{Department of Mathematics, Carleton College, Northfield, MN 55057}
\email{kulkarnin@carleton.edu}
\author{Joe~Suk}
\address{Department of Mathematics, Stony Brook University, Stony Brook, NY 11794}
\email{ybjosuk@gmail.com}
\author{Ewin~Tang}
\address{Department of Mathematics, University of Texas at Austin, 2515 Speedway,
Austin, TX 78712}
\email{ewin@utexas.edu}

\theoremstyle{plain}
\newtheorem{theorem}{Theorem}[subsection]
\let\oldthm\theorem
\renewcommand{\theorem}{\oldthm\normalfont}
\newtheorem*{theorem*}{Theorem}

\newtheorem{proposition}[theorem]{Proposition}
\newtheorem*{proposition*}{Proposition}
\newtheorem{corollary}[theorem]{Corollary}
\newtheorem*{corollary*}{Corollary}
\newtheorem{lemma}[theorem]{Lemma}
\newtheorem*{lemma*}{Lemma}
\let\oldcor\corollary
\renewcommand{\corollary}{\oldcor\normalfont}
\let\oldlemma\lemma
\renewcommand{\lemma}{\oldlemma\normalfont}
\theoremstyle{definition}
\newtheorem*{definition}{Definition}

\newtheorem{remark}[theorem]{Remark}

\numberwithin{equation}{subsection}

\begin{document}
\begin{abstract}
    A matrix is \emph{$k$-nonnegative} if all its minors of size $k$ or less are nonnegative.
    We give a parametrized set of generators and relations for the semigroup of $k$-nonnegative $n\times n$ invertible matrices in two special cases: when $k = n-1$ and when $k = n-2$, restricted to unitriangular matrices.
    For these two cases, we prove that the set of $k$-nonnegative matrices can be partitioned into cells based on their factorizations into generators, generalizing the notion of Bruhat cells from totally nonnegative matrices.
    Like Bruhat cells, these cells are homeomorphic to open balls and have a topological structure that neatly relates closure of cells to subwords of factorizations.
    In the case of $(n-2)$-nonnegative unitriangular matrices, we show the cells form a Bruhat-like CW-complex.
\end{abstract}
\maketitle
\section{Introduction}

A totally nonnegative (respectively totally positive) matrix is a matrix whose minors are all nonnegative (respectively positive).
Total positivity and nonnegativity are well-studied phenomena and arise in diverse areas such as planar networks, combinatorics, and stochastic processes \cite{fomin-zel}.
We generalize the notion of total nonnegativity and positivity as follows.
A \emph{$k$-nonnegative} (resp. \emph{$k$-positive}) matrix is a matrix where all minors of order $k$ or less are nonnegative (resp. positive).

Our investigation of $k$-nonnegative matrices follows the path of previous work in parametrizing totally nonnegative matrices.
In their 1998 paper \cite{fomin-zel-bruhat}, Fomin and Zelevinsky partition the semigroup of invertible totally nonnegative matrices into cells based on their factorizations into parametrized generators.
These cells form a regular CW-complex whose closure poset is isomorphic to the Bruhat poset on words corresponding to factorizations of matrices in cells.
These amazing and elegant results arise from viewing totally nonnegative matrices in the context of the theory of Lie groups.
The generality to which these results hold motivates us to investigate the extent to which a similar structure can be found in semigroups of $k$-nonnegative matrices.

The idea of $k$-nonnegativity is not new; in fact, $k$-nonnegative and $k$-positive matrices have been the subject of study in several papers by Fallat, Johnson, and others (\cite{fallat-johnson, fallat-johnson-sokal}).
However, these works are largely unconcerned with the semigroup structure of $k$-nonnegative matrices and take considerably different directions than our own approach to understanding these matrices.
This work is, to the authors' knowledge, the first attempt to fully characterize the generators of this semigroup and find an analogous Bruhat cell decomposition as done by Fomin and Zelevinsky in \cite{fomin-zel-bruhat} for the case of totally nonnegative matrices.

Our results are as follows.
We begin by generalizing the Loewner-Whitney theorem to $k$-nonnegative $n\times n$ matrices.
This allows us to compute generators for two semigroups: $(n-1)$-nonnegative invertible matrices and $(n-2)$-nonnegative unitriangular matrices.
In both cases, our generating sets consist of the generators for the corresponding totally nonnegative case (elementary Jacobi matrices for the $(n-1)$ case and unitriangular Chevalley generators for the $(n-2)$ case), as well as a new class of generators, which we call $K$-generators in the $n-1$ case, and $T$-generators in the $n-2$ case.
We will show that it is possible to factor any matrix in the semigroup as a product of these generators, and we also give relations by which we can move between any two factorizations of the same matrix.

As done by Fomin and Zelevinsky \cite{fomin-zel}, we proceed to group matrices into cells based on their factorizations into generators. We show that these cells partition the space and behave well with respect to closure and their parameters; namely, the parametrization of these cells is a homeomorphism and we can obtain all of the elements in the closure of a cell by setting elements in the parametrization to zero (or equivalently, by taking subwords of our factorization word).
In fact, in the $n-2$ case, the cells form a CW-complex.
Surprisingly, throughout this process, the generators, relations, cells, and resulting closure poset can all be simply described and follow fairly naturally from the restrictions on the space.
This suggests that such structure might exist for more cases and even for $k$-nonnegative matrices in general.

Our paper is structured as follows.
In Section~\ref{sec:preliminaries}, we detail some relevant background and proceed to our most general results on factorizations of $k$-nonnegative matrices.
We describe partial factorizations of $k$-nonnegative matrices and give some lemmas necessary for future sections.
Section~\ref{sec:factorization} describes our two special cases: $(n-1)$-nonnegative matrices and $(n-2)$-nonnegative unitriangular matrices. We give specific generating sets for the corresponding semigroups as well as sets of relations.
In Section~\ref{sec:bruhat}, we describe the resulting cells in these two cases, and discuss the topological properties that these cells share with the standard Bruhat cells of totally nonnegative matrices.

\section{Preliminaries}\label{sec:preliminaries}
\subsection{Background}\label{subsec:background}

We begin by establishing some conventions and notation that will be used throughout the paper. We use $[n]$ to refer to the set $\{1,\ldots,n\}$. For any matrix $X$, $X_{I,J}$ refers to the submatrix of $X$ indexed by a subset of its rows $I$ and a subset of its columns $J$, and $\left| X_{I,J}\right|$ will refer to the minor, that is, the determinant of this submatrix. We say a minor $\left|X_{I,J}\right|$ is of \textit{order} $k$ if $|I|=|J|=k$. A minor $\left|X_{I,J}\right|$ is called \emph{solid} if both $I$ and $J$ are intervals. $|X_{I,J}|$ is called \emph{column-solid} if $J$ is an interval.
Unless stated otherwise, all matrices discussed will belong to $GL_n(\mathbb{R})$. We will sometimes abbreviate totally nonnegative as TNN.

The set of all invertible $k$-nonnegative $n\times n$ matrices is closed under multiplication, as can be seen using the Cauchy-Binet formula, and thus forms a semigroup. Similarly, the set of all upper unitriangular $k$-nonnegative matrices form a semigroup (and analogously for lower unitriangular). We would like to study the structure and topology of these semigroups. What are the generators and relations of the semigroup? What topological features does it have when endowed with the standard topology on $GL_n(\mathbb{R})$? We first summarize here the known answers to these questions for the case of $k=n$ (i.e. totally nonnegative matrices). Most of the following summary can be found in \cite{fomin-zel}.

We first discuss the generators of the semigroup of totally nonnegative matrices. A \emph{Chevalley generator} is defined as a matrix with one of the two following forms: matrices that only differ from the identity by having some $a>0$ in the $(i,i+1)$-st entry, denoted $e_i(a)$, and the transpose form, denoted $f_i(a) = e_i(a)^T$. The name ``Chevalley generator'' comes from the fact that $e_i(a)$ generate the semigroup of upper unitriangular totally nonnegative matrices. More generally, \emph{elementary Jacobi matrices} differ from the identity in exactly one positive entry either on, directly above, or directly below the main diagonal. Thus, an elementary Jacobi matrix is a Chevalley generator or a diagonal matrix which differs from the identity in one entry on the main diagonal.

The Loewner-Whitney Theorem (Theorem 2.2.2 of \cite{TNNbook}) states that the elementary Jacobi matrices generate the semigroup of invertible totally nonnegative matrices. Furthermore, one can show that any upper unitriangular matrix $X$ can be factored into Chevalley generators $e_i(a)$ with nonnegative parameters $a\geq 0$. Similarly, any lower unitriangular matrix $X$ can be factored into Chevalley generators $f_i(a)$ with nonnegative parameters $a\geq 0$. These factorizations will allow us to parametrize the entire semigroup.

To do this, we need to discuss Bruhat cells. The following will come primarily from \cite{fomin-zel-bruhat} \S~4.  Let us establish some notation.
We will let $G = GL_n(\mathbb{R})$, although many of the results here (as well as notions of total nonnegativity and total positivity) hold in the broader context where $G$ is any semisimple group.
Let $B^+$ (resp. $B^-$) be the subgroup of upper-triangular (resp. lower-triangular) matrices in $G$.
We can identify $S_n$ with a subgroup of $G$ by identifying an $\omega \in S_n$ with the matrix sending the basis vector $e_i$ to the basis vector $e_{\omega(i)}$.

For any $u\in S_n$, let $B^+uB^+$ (resp. $B^-uB^-$) denote the corresponding double coset. We have decompositions
    \[ G = \bigcup_{u \in S_n} B^+uB^+ = \bigcup_{v \in S_n} B^-vB^- \]
We call a particular double coset $B^+uB^+$ or $B^-vB^-$ a \emph{Bruhat cell} of $G$. We then define \emph{double Bruhat cells} as
    \[ B_{u,v} = B_u^+ \cap B_v^- := B^+uB^+ \cap B^-vB^- \]
so that $G$ is partitioned into these $B_{u,v}$ for $(u,v) \in S_n \times S_n$.

Next, let
$$\mathbb{A}=\langle 1,\ldots,n-1,\circled{1},\ldots,\circled{$n$},\overline{1},\ldots,\overline{n-1}\rangle $$
and let
\begin{align*}
x_{\circled{\tiny $i$}}(t)&=h_i(t)\\
x_i(t)&=e_i(t)\\
x_{\overline{i}}(t)&=f_i(t)
\end{align*}
Thus, the generators of $\mathbb{A}$ correspond to the different types of generators for the semigroup. For any word $\textbf{i}:=(i_1,\ldots,i_{\ell})\in\mathbb{A}$, there is a product map $x_{\textbf{i}}:\mathbb{R}_{>0}^{\ell}\to G$ defined by
$$x_{\textbf{i}}(t_1,\ldots,t_{\ell}):=x_{i_1}(t_1)\cdots x_{i_{\ell}}(t_{\ell})$$
With some conditions imposed on $\textbf{i}$, it turns out that the image of this map describes precisely the totally nonnegative matrices in a particular double Bruhat cell, allowing us to parametrize the double Bruhat cell and, consequently, the semigroup of totally nonnegative matrices. We describe this in more detail by introducing a definition:

\begin{definition}
Let $u,v\in S_n$. A \emph{factorization scheme} of type $(u,v)$ is a word $\textbf{i}$ of length $n+\ell(u)+\ell(v)$ (where $\ell(u)$ denotes the Bruhat length of $u$ in $S_n$) in $\mathbb{A}$ such that the subword of barred (resp. unbarred) entries of $\textbf{i}$ form a reduced word for $u$ (resp. $v$) and such that each circled entry $\circled{$i$}$ is contained exactly once in $\textbf{i}$.
\end{definition}

Next, we have the main result which allows us to parametrize totally nonnegative matrices. For the case of upper unitriangular matrices, it suffices to consider Bruhat cells:

\begin{theorem}[Theorems 2.2.3, 5.1.1, 5.1.4, and 5.4.1 of \cite{Berenstein-Fomin-Zelevinsky}]
Let $N_{\geq 0}$ be the set of $n\times n$ upper unitriangular totally-nonnegative matrices. Then, $N_{\geq 0}\cap B_w^-$ partition $N_{\geq 0}$ as $w$ ranges over $S_n$. Furthermore, each $N_{\geq 0}\cap B_w^-$ is in bijective correspondence with an $\ell(w)$-tuple of positive real numbers via the map $(t_1,\ldots, t_{\ell(w)})\mapsto e_{h_1}(t_1)\cdots e_{h_{\ell(w)}}(t_{\ell(w)})$ where $(h_1,\ldots,h_{\ell(w)})$ is a reduced word for $w$.
\end{theorem}

We shall later refer to the image of the product map $(t_1,\ldots,t_{\ell(w)})\mapsto e_{h_1}(t_1)\cdots e_{h_{\ell(w)}}(t_{\ell(w)})$ as $U(w)$. A similar result holds for the general case:

\begin{theorem}[Theorems 4.4 and 4.12 in \cite{fomin-zel-bruhat}]\label{thm:tnn-decomp}
If $\textbf{i}=(i_1,\ldots,i_{\ell})$ is a factorization scheme of type $(u,v)$, then the map $x_{\textbf{i}}$ is a bijection between $\ell$-tuples of positive real numbers and totally nonnegative matrices in the double Bruhat cell $B_{u,v}$.
\end{theorem}

We do not include the proof of the above two theorems here, but one key to proving that the map $x_{\textbf{i}}$ is a bijection is understanding the commutation relations between the generators of the semigroup. The commutation relations tell us how to move between two different factorizations of a totally nonnegative matrix into elementary Jacobi generators.

For the semigroup of upper unitriangular totally nonnegative matrices, the relations obeyed by the $e_i$'s are the same as the braid relations between adjacent transpositions in the Coxeter presentation of the symmetric group. More precisely, any factorization of an upper unitriangular totally nonnegative matrix into Chevalley generators $e_{i_1}(t_1)\cdots e_{i_{\ell}}(t_{\ell})$ is equal to a factorization $e_{i_1'}(t_1')\cdots e_{i_{\ell}'}(t_{\ell}')$ where the corresponding words $(i_1,\ldots,i_{\ell})$ and $(i_1',\ldots,i_{\ell}')$ differ by a braid move and the parameters $t_1',\ldots,t_{\ell}'$ can be given by invertible, rational, subtraction-free expressions in $t_1,\ldots,t_{\ell}$. The case of lower unitriangular totally nonnegative matrices with Chevalley generators of the form $f_i$ is similar.

For the case of totally nonnegative square matrices, the commutation relations between the $h_i$'s, $f_i$'s, and $e_i$'s give enough information to move between equivalent factorizations. One can observe that any totally nonnegative square matrix admits an LDU factorization, meaning that it can be factored into elementary Jacobi matrices where the $f_i$'s are grouped together on the left, the $e_i$'s are on the right, and the $h_i$'s are in between. Once we understand the commutation relations between the elementary Jacobi matrices of different types, we may perform moves to take any factorization to a factorization which arises via an LDU decomposition so that the uniqueness of the factorization, up to these moves, follows from the unitriangular cases. These necessary relations are not shown in this paper, but can be found in Section~2.2 and Theorem~1.9 of \cite{fomin-zel-bruhat}.

Next, we discuss the topology of the semigroups. The Bruhat cells give a stratification of the semigroup of upper unitriangular totally nonnegative matrices. The corresponding poset of closure relations is isomorphic to the poset induced by the Bruhat order on $S_n$. Similarly, the poset of closure relations in the double Bruhat cells of the semigroup of totally nonnegative matrices is isomorphic to the poset induced by the Bruhat order on the Coxeter group $S_n\times S_n$. As a result, many of the nice properties of these Bruhat order posets transfer to the double Bruhat and bruhat decompositions of totally nonnegative and unitriangular totally nonnegative matrices, respectively.

We will use several classic properties of this poset to deduce analogous results for the decomposition of the semigroup of unitriangular $(n-2)$-nonnegative matrices in Section \ref{subsec-n-2}. First, underlying all of our theory is the idea that we can compare elements of the group in the Bruhat order using their corresponding reduced words; two words $u,w\in S_n$ satisfy $u\leq w$ if and only if a subword of any reduced expression for $w$ is a reduced expression for $u$ (this is known as the Subword Property, or Theorem 2.2.2 in \cite{bjorner-brenti}). Another important fact is the Exchange Property (Theorem 1.5.1 of \cite{bjorner-brenti}), which states that, for any $s=s_m$ with $m\in[1,n-1]$ and $w=s_1s_2\cdots s_k\in S_n$, $\ell(sw)\leq \ell(w)\implies sw=s_1\ldots \hat{s_i}\ldots s_k$ for some $i\in [k]$. 

The posets of closure relations on cells of totally nonnegative matrices (resp. unitriangular totally nonnegative matrices) has many special properties: it has a top and bottom element (Proposition 2.3.1 of \cite{bjorner-brenti}), it is ranked (Theorem 2.2.6 of \cite{bjorner-brenti}), and it is Eulerian. In addition, for the case of unitriangular totally nonnegative matrices, the corresponding poset is a regular CW-complex, i.e.\ the closure of each Bruhat cell is homeomorphic to a closed ball \cite{Hersh2014}. Our work will show how far these properties extend to cells of $k$-nonnegative matrices in our two special cases.

\subsection{Equivalent Conditions and Elementary Generalizations}
  
When discussing $k$-nonnegative matrices, it is useful to ask first whether we need to check all minors for nonnegativity (usually an intractable computation), or just some subset of minors. For example, a well-known result, from \cite{fomin-zel}, is that total nonnegativity can be determined by checking only column-solid minors. Brosowsky and Mason study necessary and sufficient conditions for \emph{k-positivity}, in their work in Section 4 of \cite{REUreport}.

The following statement, which follows from Fallat and Johnson \cite{TNNbook}, provides a sufficient condition for $k$-nonnegativity.

\begin{theorem} \label{thm:col-solid}
   An invertible matrix $X$ is $k$-nonnegative if all column-solid (or alternatively, row-solid) minors of $X$ of order $k$ or less are nonnegative.
\end{theorem}
\begin{proof}
    Let $Q_n(q) = (q^{(i-j)^2})_{i,j=1}^n$ for $q \in (0,1)$.
    This matrix is totally positive and satisfies $\lim_{q \to 0^+} Q_n(q) = I_n$.
    Let $X_q = Q_n(q)X$, and apply the Cauchy-Binet formula on an order $r \leq k$ column-solid minor:
    \begin{align*}
        \left|(X_q)_{I,J}\right| &= \sum_{\substack{S \subset [n]\\ |S| = r}} \left|Q_n(q)_{I,S}\right|\left|X_{S,J}\right|
    \end{align*}
    The sum on the right hand side must be positive, since the column-solid minors of $X$ are nonnegative and $X$ is invertible.
    By Corollary 3.1.6 of \cite{TNNbook} (which gives an equivalent condition for $k$-positivity), $X_q$ must be $k$-positive. Taking the limit $q \to 0^+$, we see that $X$ is $k$-nonnegative.
    To get the analogous statement for row-solid minors, we can use $X_q = XQ_n(q)$.
\end{proof}

With this condition, we are sufficiently equipped to turn to factorizations of $k$-nonnegative matrices.
To begin, we would like to know the extent to which these matrices can be factored into elementary Jacobi matrices.
We can consider this as an algebraic problem of factoring divisors in the semigroup.
This leads us to the notion of \emph{$k$-irreducibility}.
\begin{definition}
    A $k$-nonnegative matrix $M$ is $k$-irreducible if $M = RS$ in the semigroup of invertible $k$-nonnegative matrices implies $R,S \notin \{f_i(a), e_i(a) \mid a > 0\}$.
\end{definition}
This definition will lead the reader to expect the following theorem.
\begin{theorem} \label{thm:factorization}
    Every $k$-nonnegative matrix $X$ can be factored into a product of finitely many Chevalley generators and a $k$-irreducible matrix.
\end{theorem}
\begin{proof}
    Suppose $X$ is not $k$-irreducible.
    Then we can suppose without loss of generality that $e_i(a)^{-1}X$ is $k$-nonnegative for some $i\in[n]$ and $a\in\mathbb{R}_{>0}$ (corresponding to removing $a$ copies of row $i+1$ from row $i$).
    We claim it is possible to choose $b>0$ so that $e_i(b+\delta)^{-1}X$ is not $k$-nonnegative for any $\delta>0$.

    We want to determine when $e_i(x)^{-1}X$ is $k$-nonnegative in terms of $x$.
    It suffices to consider row-solid order $\leq k$ minors containing row $i$ and not row $i+1$, which we will index by $\gamma$.
    These determinants $d_{\gamma}$ are multilinear functions in the rows so that $d_{\gamma}(x)=|A|-x |B|$ for some submatrices $A,B$ of $X$.
    Thus, $d_\gamma^{-1}([0,\infty))$ is closed for any $\gamma$.  There must be some $d_\gamma$ such that $d_{\gamma}^{-1}([0,\infty))$ is bounded, as otherwise we would break invertibility of $X$ since the minors $|B|$ induced by the $d_{\gamma}$'s cannot all be $0$. Thus, the intersection $\cap_\gamma d_\gamma^{-1}([0,\infty))$ is closed and compact.
    We know this set is nonempty because $a$ is in it.
    Thus, there is a maximal $b$ in the intersection and applying an inverse Chevalley with any greater value will make the product matrix not $k$-nonnegative.
    It is also clear that this maximal $b$ is of the form $|A|/|B|$ where there is some $d_\gamma=|A|-x|B|$.  In fact, we must have $b=\min_\gamma\{|A|/|B|\ |\ d_\gamma=|A|-x|B|, |B|\neq 0\}$.
    Thus, $e_i(b+\delta)^{-1}X$ is not $k$-nonnegative for any $\delta>0$.

    So, in this way, we factor out a Chevalley generator, leaving a matrix with one more zero minor of order at most $k$.
    We can iterate this process, which must stop eventually because the number of minors of size at most $k$ is finite.
    The resulting matrix must be $k$-irreducible.
\end{proof}
Thus, $k$-irreducible matrices and Chevalley matrices will form a generating set for the semigroup of $k$-nonnegative matrices.\footnote{This does not give a minimal generating set, as there are $k$-irreducible matrices that can be factored as $Xe_iY$ where $X,Y$ are $k$-irreducible: see the case of tridiagonal matrices in~\cite{REUreport}.}

The problem of describing $k$-irreducible matrices seems to be difficult in general.
We investigate the shape of these matrices in general in our report \cite{REUreport}, where we consider locations of zero minors and extend Section~7.2 of \cite{TNNbook}.
In this work, we will focus on the results that allow us to give a complete characterization of $k$-irreducible matrices in our two special cases.

To this end, we first consider where zero entries can occur in $k$-nonnegative matrices.
\begin{lemma} \label{lem:zero-shadow}
    Let $X$ be a $k$-nonnegative invertible matrix, for $k \geq 2$.
    If $x_{ij} = 0$, then:
    \begin{enumerate}
        \item $i \neq j$,
        \item if $i < j$, then $x_{i'j'} = 0$ for $i' \leq i$, $j' \geq j$,
        \item if $i > j$, then $x_{i'j'} = 0$ for $i' \geq i$, $j' \leq j$.
    \end{enumerate}
\end{lemma}
This immediately follows from considering the $2\times 2$ submatrices containing $x_{ij}$ and the invertibility of $X$.

Now, we describe the extent to which we can factor Chevalley matrices from a generic $k$-nonnegative matrix. First, it will be useful to explicitly state the following technical lemma fully describing the minors of the product of a matrix with a Chevalley generator. It immediately follows from Cauchy-Binet.
\begin{lemma}\label{lem:tech}
    Suppose $X'$ is a matrix $X$, transformed through left or right multiplication by a Chevalley generator.
    Then minors of $X'$ are equal to minors of $X$ except in the following cases:
    \begin{enumerate}
        \item If $X' = X e_k(a)$ (adding $a$ copies of column $k$ to column $k+1$) then $\left|X'_{I,J}\right| = \left|X_{I,J}\right| + a\left|X_{I,J\setminus k+1 \cup k}\right|$ when $J$ contains $k+1$ but not $k$;
        \item If $X' = e_k(a) X$ (adding $a$ copies of row $k+1$ to row $k$) then $\left|X'_{I,J}\right| = \left|X_{I,J}\right| + a\left|X_{I\setminus k \cup k+1, J}\right|$ when $I$ contains $k$ but not $k+1$.
    \end{enumerate}
    For $f_k$, swap $k+1$ and $k$ in the above statements.
\end{lemma}
Now, we can show that $k$-irreducible matrices have staircases of zeroes in their northeast and southwest corners.
\begin{theorem}\label{thm:decomp}
    If an invertible matrix $A$ is $k$-nonnegative, we can express it as a product of Chevalley matrices (specifically, only $e_i$s) and a single $k$-nonnegative matrix where the $ij$-th entry is zero when $|j-i| > n-k$ .

    That is, if a matrix is $k$-irreducible, the $ij$-th entry is zero when $|j-i| > n-k$.  Note that if we set $k = n$, we get the Loewner-Whitney theorem.
\end{theorem}
\begin{proof}
    We use the following lemma:
    \begin{lemma} \label{lem:entry-zeros}
        Let $A$ be an invertible $k$-nonnegative matrix. Let $i$ and $j$ be indices such that (1) $i < j$, (2) $a_{x,y} = 0$ for $x \leq i+1$ and $y \geq j$ except for $a_{i+1,j}$ and $a_{i,j}$, which have no restrictions on them, and (3) $i<k$. Then either $a_{i,j} = 0$ or $e_i(-a_{i,j}/a_{i+1,j})A$ is $k$-nonnegative.

        (The analogous lemma holds in the transpose case, using $f_i$s, and also when factoring a Chevalley generator from the right of $A$.)
    \end{lemma}
    In the latter case, we can reduce our matrix to one where $a_{ij}$ is zero by factoring out a Chevalley matrix.
    \begin{proof}
        First, notice that our row operation is well-defined, since $a_{i+1,j} = 0 \implies a_{i,j} = 0$ from Lemma~\ref{lem:zero-shadow} and (1).
        Further, notice that from Theorem~\ref{thm:col-solid} and Lemma~\ref{lem:tech} the only minors we need to worry about are those row-solid minors containing row $i$ but not row $i+1$.
        From (2), it suffices to consider any $I,J$ where $I = [h,i]$ for some $h<i$ and $J$ has no indices greater than or equal to $j$ to check for nonnegativity.
        Let $I,J$ define such a minor. Then using Lemma~\ref{lem:tech},
        \begin{align*}
            \left|e_i(-a_{i,j}/a_{i+1,j})A_{I,J}\right| &= \left|A_{I,J}\right| - \frac{a_{i,j}}{a_{i+1,j}}\left|A_{I\setminus i \cup i+1,J}\right| \\
            &= \frac{1}{a_{i+1,j}}\left(a_{i+1,j}\left|A_{I,J}\right| - a_{i,j}\left|A_{I\setminus i \cup i+1,J}\right|\right) \\
            &= \frac{1}{a_{i+1,j}} \left|A_{I\cup i+1,J \cup j}\right|
        \end{align*}
        And because the minor $|A_{U\cup i+1,J\cup j}|$ is of order one greater than the order of the original minor $|A_{I,J}|$, we have that the resulting matrix must be $k$-nonnegative since $|A_{I,J}|$ is of order less than $k$ because of (3).
    \end{proof}
    We can iterate this factorization, using the criteria to find another entry to eliminate.
    The top-right corner satisfies the criterion for the lemma, and for a matrix where that entry is zero, the entry directly below satisfies the criterion, and so on.
    We can eliminate $k-1$ entries in the last column, one by one top-down, then $k-2$ entries in the second-to-last, and continue until all entries desired are zero.
    For the second case $i>j$, one can consider the transpose of the matrix and use the above argument to get the zeros in the bottom-left corner of the original matrix.
\end{proof}

A final remark on zero entries in $k$-irreducible matrices results from a simple application of Lemma \ref{lem:entry-zeros}.
\begin{corollary} \label{cor:more-decomp}
    Let $M$ be a $k$-nonnegative $k$-irreducible matrix for $k > 2$. Then
    \begin{align*}
        m_{i,j} = 0 \implies \begin{cases}
        m_{i-1,j-1} = 0 & \text{if } i \leq k \text{ or } j \leq k \\
        m_{i+1,j+1} = 0 & \text{if } i > n-k \text{ or } j > n-k
    \end{cases}
    \end{align*}
\end{corollary}

\section{Factorizations}\label{sec:factorization}

In this section, we describe generators for the semigroup of invertible $(n-1)$-nonnegative matrices and for the semigroup of unitriangular $(n-2)$-nonnegative matrices. We also give sets of relations for these matrices, on the basis of which we will construct our Bruhat cell analogues.

As motivation for these choices of semigroups, we notice that, as a result of Theorem~$\ref{thm:decomp}$, the smaller $k$ is, the less the semigroup of $k$-nonnegative matrices seems to resemble the structure of the semigroup of totally nonnegative matrices.
Note that $(n-1)$-nonnegative unitriangular matrices must be totally nonnegative.

\subsection{The \texorpdfstring{$k=n-1$}{k=n-1} Case} 

Let $M$ be an invertible $(n-1)$-irreducible matrix.
Notice that by Theorem~\ref{thm:decomp}, the only nonzero entries of $M$ must be on the diagonal, subdiagonal or superdiagonal.
Further, by Lemma~\ref{lem:zero-shadow} and Corollary~\ref{cor:more-decomp}, if $M$ has a zero entry in any of the above three diagonals, it must be totally nonnegative, and by the Loewner-Whitney theorem, diagonal.
Thus, $M$ must have all nonzero elements on its three diagonals. Without loss of generality we can normalize the subdiagonal to ones, leaving $2n-1$ entries that can vary; we notate these entries by $a_i = M_{i,i}$ and $b_i = M_{i,i+1}$. 

We will show that matrices of this form are $(n-1)$-irreducible precisely when certain key minors in the matrix are zero. In particular, we must impose more restrictions on the matrix entries to make said key minors zero.

We first observe that a minor in a tridiagonal matrix (where entries on the subdiagonal are ones) can be expressed in terms of a continued fraction.
We will notate continued fractions in the following way:
\begin{align*}
    \left[a_0;a_1,\ldots,a_m;b_1,\ldots,b_m\right] &:= a_0 - \frac{b_1}{a_1 - \frac{b_2}{a_2 - \cdots}}
\end{align*}
This is different from the standard notation, which adds recursively rather than subtracts. Let $C_i(j) = \left|M_{[i,i+j-1],[i,i+j-1]}\right|$. Then the following recursive relation is satisfied:
\begin{equation}\label{eqn:recurrence}
    C_i(0) = 1,\; C_i(1) = a_{i},\; C_i(r) = a_{i+r-1}C_i(r-1) - b_{i+r-2}C_i(r-2) \tag{$\star$}
\end{equation}
This is sometimes known as the recurrence defining the \emph{generalized continuant}.
The above statement is rephrased slightly in the following lemma.
\begin{lemma} \label{lem:continued-fraction}
    $C_i(r) = C_i(r-1)[a_{i+r-1};a_{i+r-2},\ldots a_i;b_{i+r-2},\ldots,b_i]$ when $C_i(k) \neq 0$ for $k < r$.
\end{lemma}
\begin{proof}
    It is obviously true for the base cases of the recurrence.
    Rewrite the equation as follows:
    \begin{align*}
        \frac{C_i(r)}{C_i(r-1)} &= a_{i+r-1} - b_{i+r-2}\frac{C_i(r-2)}{C_i(r-1)} \\
        & = a_{i+r-1} - \frac{b_{i+r-2}}{[a_{i+r-2};a_{i+r-3},\ldots,a_i;b_{i+r-3},\ldots,b_i]} \\
        & = [a_{i+r-1};a_{i+r-2},\ldots a_i;b_{i+r-2},\ldots,b_i]
    \end{align*}
\end{proof}
These recurrences also hold if we take the base case to be at the bottom corner rather than the top corner, thus relating $C_i(r)$ to $C_{i+1}(r-1)$.
To relate these recurrences to nonnegativity tests, we use the following lemma, which can be shown through direct computation.
\begin{lemma} \label{lem:tridiag-minors}
    Nonzero column-solid minors of tridiagonal matrices evaluate either to a product of entries on the same diagonal, or to a solid minor multiplied by a product of entries on the same diagonal.
\end{lemma}

Now, we can consider how the irreducibility and nonnegativity conditions restrict the possible values of matrix entries.
\begin{theorem}
    Let $M$ be a tridiagonal matrix with 1s on the subdiagonal and nonzero entries on the diagonal and superdiagonal.
    Then $M$ is invertible, $(n-1)$-nonnegative, and $(n-1)$-irreducible if and only if the following hold:
    \begin{align*}
        b_i & > 0  \qquad \text{for all } i\\
        [a_x;a_{x-1},\ldots a_1;b_{x-1},\ldots,b_1] & > 0 \qquad \text{for } x < n-1 \\
        [a_{n-1};a_{n-2},\ldots a_{1};b_{n-2},\ldots,b_{1}] & = 0 \\
        [a_{n};a_{n-1},\ldots a_{2};b_{n-1},\ldots,b_2] & = 0
    \end{align*}
\end{theorem}
\begin{proof}
    By Theorem \ref{thm:col-solid}, $M$ is $(n-1)$-nonnegative if and only if column-solid minors of order at most $n-1$ are nonnegative.
    This is equivalent, using Lemma \ref{lem:tridiag-minors}, to the condition that all minors of the form $C_i(j)$ are nonnegative for $j \leq n-1$, and all of the $b_i$'s being nonnegative (that is, positive, since they are assumed to be nonzero).
    
    We claim that for $j < n-1$, we cannot have $C_i(j) = 0$.
    To see this, consider the smallest such $j$. By our recurrence relation~\ref{eqn:recurrence}, this must mean that $C_i(j+1)$ is negative, which breaks $(n-1)$-nonnegativity.
    If $C_i(j+1)$ is not a valid minor, then just take the recurrence in the opposite direction to get a contradiction for $C_{i-1}(j+1)$.
    Thus, we can use Lemma~\ref{lem:continued-fraction}, and say that $C_i(j)$ are all nonnegative precisely when the base cases, that is $a_i$'s, are positive, as well as all of the corresponding continued fractions.
    Among these continued fractions, notice that if $[a_k;a_{k-1},\ldots a_i;b_{k-1},\ldots,b_i] > 0$, then so are the continued fractions achieved by truncating on the right at any $j \leq k$.

    This gives us a necessary and sufficient condition for $k$-nonnegativity:
    \begin{align*}
        b_i & > 0 \\
        [a_x;a_{x-1},\ldots a_1;b_{x-1},\ldots,b_1] & > 0 \qquad \text{if } x < n-1 \\
        [a_{n-1};a_{n-2},\ldots a_1;b_{n-2},\ldots,b_1] & \geq 0 \\
        [a_n;a_{n-1},\ldots a_2;b_{n-1},\ldots,b_2] & \geq 0
    \end{align*}
    Second, to guarantee $k$-irreducibility, notice that the only Chevalley generators we need to worry about factoring out are $e_{n-1}$ and $f_1$ (from the left), and $e_1$ and $f_{n-1}$ (from the right).
    If we cannot factor any of these from $M$, then the order $n-1$ principal minors $C_2(n-1)$ and $C_{1}(n-1)$ are zero (if not, either the matrix form would be different from the one described or the matrix is no longer $k$-nonnegative).
    
    Finally, this matrix is invertible, as the recurrence relation clearly shows that it has a negative determinant.
\end{proof}
The above criterion can be simplified into a $(2n-3)$-parameter family.
This family, along with elementary Jacobi matrices, generates the semigroup of $(n-1)$-nonnegative matrices.\footnote{Note that we can actually reduce this to an $(n-3)$-parameter family, just by scaling the superdiagonal to ones via diagonal matrices.
However, we will lose relations between $h_i$ generators if we do so, making the corresponding cell structure much more complicated.}

The $(2n-3)$-parameter family of generators appears as follows.
    \[ K(\vec{a}, \vec{b}) = \begin{bmatrix}
        a_1 & a_1b_1 \\
        1 & a_2+b_1 & a_2b_2 \\
        & 1 & \ddots & \ddots \\
        & & \ddots & a_{n-2}+b_{n-3} & a_{n-2}b_{n-2} \\
        & & & 1 & b_{n-2} & b_{n-1}Y \\
        & & & & 1 & b_{n-1}X \\
    \end{bmatrix} \]
where $a_1,\ldots, a_{n-2}, b_1, \ldots, b_{n-1}$ are positive numbers, $Y = b_1\cdots b_{n-2}$ and 
    $$X = \left| K_{[2, n-2],[2,n-2]}\right| = \sum_{k=1}^{n-2}\left(\prod_{\ell=2}^k b_{\ell-1} \prod_{\ell=k+1}^{n-2} a_\ell\right).$$
All the minors of these $K$-generators, except the determinant, can be written as subtraction-free expressions in the parameters. The full list of solid minors can be found in Appendix~\ref{app:minors}, but in particular, we will use the following:
\begin{gather}
    \left|K_{[1,n-1],[1,n-1]}\right| = \left|K_{[2,n],[2,n]}\right| = 0 \\
    \left|K\right| = -a_1a_2\cdots a_{n-2}b_1b_2\cdots b_{n-1}
\end{gather}
We can show that this set of generators is minimal, in the sense that every element in the set is necessary.
\begin{theorem} \label{thm:n-1-minimality}
    Let $M = K(\vec{a},\vec{b})$ be a $K$-generator as defined above. Then if $RS = M$ in the semigroup of invertible $(n-1)$-nonnegative $n \times n$ matrices, one of $R$ or $S$ is a diagonal matrix.
\end{theorem}
\begin{proof}
    Suppose we have $RS = M$.
    From Lemma~\ref{lem:zero-shadow} we know that $R$ and $S$ have nonzero diagonals.
    Thus, we know that $r_{i,i+2}$, $s_{i,i+2}$ and their transpose analogues are all 0 from the formula for matrix multiplication.
    Further, we know that for every $i$, one of $r_{i,i+1}$ and $s_{i+1,i+2}$ is 0, and one of $r_{i,i+1}$ and $s_{i,i+1}$ is positive.
    Together, these show that $R$ and $S$ can only be as described above.
\end{proof}

We now turn to relations involving generators of the form $K(\vec{a},\vec{b})$.
It can be seen by direct computation that the following relations hold.
\begin{align} \label{n-1-rels}
    e_i(x) K(\vec{a},\vec{b}) &= K(\vec{A},\vec{B}) e_{i+1}(x'), \text{ where } 1 \leq i \leq n-2 \\
    e_{n-1}(x) K(\vec{a},\vec{b}) &= K(\vec{A},\vec{B}) f_{n-1}(x') h_{n-1}(c) \\
    f_{i+1}(x) K(\vec{a},\vec{b}) &= K(\vec{A},\vec{B}) h_{i+1}(1/w) f_i (x) h_{i}(w), \text{ where } 1 \leq i \leq n-2 \\
    f_{1}(x) K(\vec{a},\vec{b}) &= K(\vec{A},\vec{B}) e_{1}(x') h_1(c) \\
    h_i(x) K(\vec{a}, \vec{b}) &= K(\vec{A}, \vec{B}) h_{i-1}(x), \text{ where } 2 \leq i \leq n \\
    h_1(x) K(\vec{a}, \vec{b}) &= K(\vec{A}, \vec{B}) \\
    K(\vec{a},\vec{b})h_n(x) &= K(\vec{A},\vec{B})
\end{align}

The values of the parameters of the relations can be found in Appendix~\ref{app:rels}. The expressions for new parameters are always subtraction-free rational expressions of the old parameters.
Thus, just like the relations between elementary Jacobi matrices, past showing equality of factorizations with parameters, they show equality of factorizations, considered without parameters and just as words (that is, as sets of matrices with the same factorization word).

Finally, one relation is missing: this is the relation involving products of $K$-generators.
The characterization of Bruhat cells in Lemma~\ref{bruhat-description} shows that such a relation exists.
However, the following lemma shows that ignoring this relation will not affect our discussion.
\begin{lemma} \label{lem:two-kays}
    $K(\vec{A},\vec{B})K(\vec{C},\vec{D})$ is totally nonnegative and can always be written in a factorization that contains fewer instances of $K$ and has no more Chevalley generators than $K\cdot K$ has parameters.
\end{lemma}
\begin{proof}
    Since the only minor that is negative is the full determinant, the product of two $K$s must be TNN.
    Further, it is pentadiagonal; any pentadiagonal TNN matrix can be factored using at most $4n-6$ parameters (this can be seen via Lemma~\ref{bruhat-description}).
    Thus, the corresponding reduced word in $e_i$s and $f_i$s has length between $4n-8$ and $4n-6$.
    $KK$ has length $4n-6$, so we can always find a shorter word using only Chevalley generators.
\end{proof}

\subsection{The \texorpdfstring{$k=n-2$}{k=n-2} Case} 

And now a similar analysis for $(n-2)$-nonnegative unitriangular matrices.
By Theorem~\ref{thm:decomp}, we know that any $(n-2)$-irreducible matrix here must only have three bands that are nonzero: the diagonal (which consists of ones), the super-diagonal, and the super-super-diagonal.
We will refer to matrices of this form as \emph{pentadiagonal unitriangular} matrices.
Lemma~\ref{lem:zero-shadow} and Corollary \ref{cor:more-decomp} tell us that if the matrices are not diagonal, all of the entries in this band of three diagonals must be nonzero.
Because we are working with a diagonal band, we can use the same recurrences as in the $k = n-1$ case, ignoring the first column and last row.
We will present the theorems necessary to get our generators with notation similar to what was used before. 
Entries on the superdiagonal are $a_i$'s, and entries on the super-superdiagonal are $b_i$'s.
\begin{theorem}
    Let $S$ be a pentadiagonal unitriangular matrix with nonzero entries on the diagonal band. Then $S$ is invertible, $(n-2)$-nonnegative, and $(n-2)$-irreducible if and only if the following hold:
    \begin{align*}
        a_i, b_i & > 0 \\
        [a_x;a_{x-1},\ldots a_1;b_{x-1},\ldots,b_1] & > 0 \qquad \text{if } x < n-2 \\
        [a_{n-2};a_{n-3},\ldots a_{1};b_{n-3},\ldots,b_{1}] & = 0 \\
        [a_{n-1};a_{n-2},\ldots a_{2};b_{n-2},\ldots,b_{2}] & = 0
    \end{align*}
\end{theorem}
The proof is very similar to the analogous proof for $n-1$-nonnegative tridiagonal matrices. The only difference is that more minors must be zero in $(n-2)$-irreducible matrices, two more than for the $n-1$ case.

The above characterization for $(n-2)$-irreducible unitriangular matrices can be simplified into a $(2n-5)$-parameter family.
This family, along with $e_i(a)$ and the identity matrix, generate all $(n-2)$-nonnegative unitriangular matrices.
Matrices in our parameter family appear as follows:
    \[ T(\vec{a},\vec{b}) = \begin{bmatrix}
        1 & a_1 & a_1b_1 & \\
        & 1 & a_2 + b_1 & a_2b_2 \\
        & & 1 & \ddots & \ddots & \\
        & & & \ddots & a_{n-3}+b_{n-4} & a_{n-3}b_{n-3} \\
        & & & & 1 & b_{n-3} & b_{n-2}Y \\
        & & & & & 1 & b_{n-2}X \\
        & & & & & & 1
    \end{bmatrix} \]
where $a_1,\ldots, a_{n-3}, b_1, \ldots, b_{n-2}$ are positive numbers, $Y = b_1\cdots b_{n-3} $ and
   $$ X = \left| T_{[2, n-3],[3,n-2]}\right| = \sum_{k=1}^{n-3}\left(\prod_{\ell=2}^k b_{\ell-1} \prod_{\ell=k+1}^{n-3} a_\ell\right)$$
The minor expressions here are essentially identical to those in Appendix~\ref{app:minors}, since 
\[
T(\vec{a},\vec{b})_{[1,n-1],[2,n]} = K(\vec{a},\vec{b})
\]
The only difference between our generators is the extra row and column added.
As before, we can give a minimality condition on our generating set.

\begin{theorem} \label{thm:n-2-minimality}
    If $RS = T(\vec{a},\vec{b})$ in the semigroup of invertible $(n-2)$-nonnegative $n \times n$ matrices, one of $R$ or $S$ is the identity.
\end{theorem}
This can be seen via an argument similar to Theorem~\ref{thm:n-1-minimality}, and as a result, we know that any generating set of the semigroup must include all elements of $T$.

We now list a set of relations involving generators of the form $T(\vec{a},\vec{b})$.
It can be seen by direct computation that the following relations hold:
\begin{align} \label{n-2-rels}
    e_i(x) T(\vec{a},\vec{b}) &= T(\vec{A},\vec{B}) e_{i+2}(x'), \text{ where } 1 \leq i \leq n-3 \\
    e_{n-2}(x) T(\vec{a},\vec{b}) &= T(\vec{A},\vec{B}) e_1(x') \\
    e_{n-1}(x) T(\vec{a},\vec{b}) &= T(\vec{A},\vec{B}) e_2(x')
\end{align}

And finally, we have a slightly different relation: a  matrix $e_{n-1}(u)e_{n-2}(v) T(\vec{a},\vec{b})$ can be parametrized also in exactly one of three different ways, that are specified below.
\begin{align}\label{n-2-cellsplit}
     e_{n-1}e_{n-2} T = e_{n-2}e_{n-1}T \sqcup e_{n-2}\cdots e_1 e_{n-1}\cdots e_2 \sqcup e_{n-2}\cdots e_1 e_{n-1}\cdots e_1 
     \end{align}
The proof that Equation~\ref{n-2-cellsplit} holds, as well as the parameters of all these relations, can be found in Appendix~\ref{app:rels}.
We will see when we define Bruhat cells for this case that the fourth relation splits one cell into three based on the value of the $[1,n-1],[2,n]$ minor: the values of the minor in the three split factorizations are negative, zero, and positive respectively.
This appears to represent splitting the open ball of the cell into two open balls, separated by a ball of lower dimension.

Analogously to the $(n-1)$ case, we can prove that one more relation exists between products of $T$s, and that it can be safely ignored.
\begin{lemma} \label{lem:two-tees}
    $T(\vec{A},\vec{B})T(\vec{C},\vec{D})$ can always be written in a factorization that uses the same number of parameters or fewer, such that the factorization contains only one instance of $T$.
\end{lemma}
\begin{proof}
This proof is similar to before.
The general description of Bruhat cells, along with Cauchy-Binet, indicate that if we want to give a factorization that does not use two copies of $T$, we would use at most $4n-10$ letters to do so.
Since $T\cdot T$ has $4n-10$ parameters, we have our statement.
\end{proof}

\section{Bruhat Cells} \label{sec:bruhat}

Now that we have parametrized generating sets for the semigroup of $(n-1)$-nonnegative invertible matrices and the semigroup of $(n-2)$-nonnegative unitriangular invertible matrices, we show that these semigroups can be partitioned into \emph{cells} based on their factorizations into these generators.
In the case of totally nonnegative matrices, these cells reflect the topological structure of the whole space, and behave nicely with respect to using relations to move between factorizations and considering cells' closures.

With a view toward generalizing results from the TNN case to our two cases, we ask the following questions, which will motivate our sequence of study.
\begin{enumerate}
    \item How do we know our list of relations is complete?
    \item Can we ignore parameters when studying factorizations? That is, do our relations respect the cell structure that naturally arises from considering factorizations?
    \item Is our cell structure respected when adding or removing generators?
    Does this cell decomposition reflect the topological structure of our space?
\end{enumerate}
The answer turns out to be yes to each of these questions.

\subsection{Preliminaries}

We first present some technical lemmas and characterizations that will be used throughout the proofs that follow.

First, recall that for $w \in S_n$, the length of $w$ is equal to the number of inversions of $w$. We now give a characterization of the Bruhat order.
For $w\in S_n$, define
    $$w[i,j]:=|\{a\in [i]:w(a)\geq j\}|$$
for $i,j\in[n]$.
In other words, $w[i,j]$ counts the number of non-zero entries between the northeast corner and the $ij$th entry of the permutation matrix for $w$.

\begin{lemma}[Theorem 2.1.5 of \cite{bjorner-brenti}]\label{lem:bb-bruhat}
Let $x,y\in S_n$. Then, the following are equivalent:
    \begin{itemize}
        \item $x\leq y$
        \item $x[i,j]\leq y[i,j]$ for all $i,j\in[n]$
    \end{itemize}
\end{lemma}

We define everything here for the $B^-$ decomposition, but taking the transpose will give the same results analogously for the $B^+$ decomposition, and taking both conditions will give descriptions for the double Bruhat cells.

\begin{definition}
    For an $\omega \in S_n$, let $I,J$ be a $\omega$-NE-ideal if $I = \omega(J)$ and $(\omega(i), i) \in (I,J) \implies (\omega(j), j) \in (I,J)$ for all $j$ such that $j > i$ and $\omega(j) < \omega(i)$.

    Call $I,J$ a shifted $\omega$-NE-ideal if $I \leq I'$ and $J' \leq J$ in termwise order for some $\omega$-NE-ideal $I', J'$ where $I,J \neq I',J'$.
\end{definition}
Essentially we choose some set of entries that have ones in the permutation matrix $\omega$, and have our ideal be those rows and columns, along with the rows and columns of any ones to the northeast of any of our existing ones.
Shifted ideals are submatrices that are further to the NE than the ideals.

\begin{definition}
    Call a matrix $X$ $\omega$-NE-bounded if the following two conditions hold:
\begin{itemize}
    \item $X_{I,J} \neq 0$ if $I,J$ is a $\omega$-NE-ideal.
    \item $X_{I,J} = 0$ if $I,J$ is a shifted $\omega$-NE-ideal.
\end{itemize}
\end{definition}
For $B^+$, the analogous definitions will be called $\omega$-SW-ideals and $\omega$-SW-bounded matrices.

\begin{lemma} \label{bruhat-description}
   $M$ is in $B^-_w$ iff it is $w$-NE-bounded.
\end{lemma}
\begin{proof}
    From inspection, the matrix $w$ is $w$-NE-bounded. Further, the Cauchy-Binet formula shows that multiplying by elements in $B^-$ preserve this.
\end{proof}

Finally, since we will be considering our new generators, $K$ and $T$, multiplied by $e_i$ and $f_i$, the following will be useful to distinguish factorizations.

Define the \emph{weak left Bruhat order} (\emph{weak right Bruhat order}, respectively) on $S_n$ such that $\alpha \leq \beta$ exactly when $\beta$ can be written as a reduced word $w \alpha$ ($\alpha w$, respectively) for some $w$.
\begin{lemma} \label{lem:cell-mult}
    Consider $S_n$ as a Coxeter group generated by the transpositions $s_i=(i,i+1)$.  Let $w$ be some reduced word of $\sigma \in S_n$ such that $\alpha \leq w$ in the weak left Bruhat order.
     Then, for $M \in B_\alpha^-$, $U(w\setminus \alpha)M \subset B_w^-$.
\end{lemma}
\begin{proof}
    We argue by induction. The base case is obviously true, since $M \in B_\alpha^-$. Now suppose we are taking some $N \in B_{\beta}^-$ and considering $e_i(c)N$ such that $s_i \beta$ is a reduced word. This occurs precisely when $s_i \beta$ has more inversions than $\beta$. Considering $\beta$ in terms of one-line notation, we see this can only happen when we have $i$ before $i+1$.

    Now, we consider the $s_i \beta$-NE-ideals, and compare them to the $\beta$-NE-ideals. The ideals that do not contain rows $i$ and $i+1$ are exactly the same, as are the ideals that contain both, since in these two cases, the corresponding minors are unaffected by the $e_i$. When a $s_i \beta$-NE-ideal contains row $i+1$, it must contain row $i$, so the only remaining case are the $s_i \beta$-NE-ideals that contain only $i$. These are in bijection with the $\beta$-NE-ideals containing $i+1$ but not $i$. We know from Lemma~\ref{lem:tech} that for an $s_i \beta$-NE-ideal $I,J$,
        \[ |(e_i(c)M)_{I,J}| = |M_{I,J}| + c|M_{I\setminus i \cup i+1, J}| > 0 \]
    since the right hand side is the sum of a $\beta$-NE-ideal and a shifted $\beta$-NE-ideal.

    Now, consider a shifted $s_i \beta$-NE-ideal $I,J$. We consider the $I',J'$ from the definition (that is, the $s_i \beta$-NE-ideal such that $I \leq I'$ and $J' \leq J$). If $I'$ contains neither $i$ nor $i+1$, or if it contains both, then $I', J'$ is a $\beta$-NE-ideal as well, and $I, J$ must be a sum of shifted ideals that are also shifted with respect to $I', J'$. If $I'$ contains $i$ but not $i+1$, then the ideal swapping out $i$ for $i+1$ is a $\beta$-NE-ideal. This ideal shows that $I,J$ can be expressed as a sum of shifted $\beta$-NE-ideal minors.
\end{proof}

Finally, we will use the closure structure of the Bruhat decompositions of $GL_n(\mathbb{R})$.
\begin{lemma} \label{lem:closure-necc}
    Let $S$ be a subset of a classical Bruhat cell $U(w)$.
    Then $\overline{S}$ is contained in the disjoint union of the cells $U(w')$, where $w' \leq w$.
\end{lemma}
\begin{proof}
    By Lemma~\ref{lem:bb-bruhat}, if $u \nleq w$, there exists $(i,j)$ with $u[i,j] > w[i,j]$.
    Consider the minimal $u$-NE-ideal $R,S$ containing cell $(i,j)$.
    Then $|X_{R,S}|\neq 0$ for $X \in U(u)$, by Lemma~\ref{bruhat-description}.
    But if $X \in U(w)$, then $X_{R,S}$ is not of full rank, because it is obtained by performing row operations on a matrix of rank less than $u[i,j]$.
    Thus $|X_{R,S}| = 0$, which means $X \notin U(u)$.
    Further, since all matrices $M \in \overline{U(w)}$ also satisfy $|M_{R,S}| = 0$, $M$ is not in $U(u)$, and the statement follows.
\end{proof}

\subsection{Cells of \texorpdfstring{$(n-1)$}{(n-1)}-nonnegative matrices}\label{subsec-n-1}

Notate the semigroup of $(n-1)$-nonnegative invertible matrices as $G$.
We will, as done in the background, associate factorizations of matrices in $G$ to words in the free monoid
    $$\mathbb{S}=\langle 1,\ldots,n-1,\circled{1},\ldots,\circled{$n$},\overline{1},\ldots,\overline{n-1}, K\rangle$$
We define a length function $\ell: \mathbb{S} \to \mathbb{Z}_{\geq 0}$ by giving each letter of the alphabet a value: $(i)$ and $(\overline{i})$ are both one, \circled{$i$} values to zero, and $K$ to $2n-3$.
That is, we are counting the number of parameters of the factorization, ignoring diagonal generators.\footnote{Ignoring these is motivated by trying to determine ``distance'' from the corresponding Weyl group.}

Let $V(W)$ be the matrices that have a factorization corresponding to the word $w \in \mathbb{S}$.
This set can also be defined as the image of the parameter map $x_W$ associated to $W$, where, if $k$ is the number of parameters of $W$,
    \[ x_W: \mathbb{R}_{>0}^{k} \to G, \qquad (t_1,\ldots,t_k) \mapsto x_{W_1}(t_1)x_{W_2}(t_2)\cdots x_{W_k}(t_k) \]
where each letter corresponds to a generator, as seen in the background, and $x_K$ is the generator $K$.

Notice that the union of these cells is precisely $G$.
Also notice that the relations given by \ref{n-1-rels} and those in Theorem 4.9 of \cite{fomin-zel-bruhat} allow us to move between factorizations, and because the relations only contain subtraction-free rational expressions, we know that the same relations can be performed on all matrices with the same factorization, regardless of the values of the parameters.
Thus, we can consider using relations on factorizations in the context of words in $\mathbb{S}$.

By the above reasoning, the relations in \ref{n-1-rels} define an equivalence relation on words in $\mathbb{S}$.
From here on, we will only be considering equality up to this equivalence relation.
As is clear by definition, if $A = B$, then $V(A) = V(B)$.
We wish to know whether $V(A) = V(B)$ implies $A = B$, that is, do we have all the relations?

First, we will give a description of exactly what distinct cells we have under these relations via a characterization of their reduced words.
Notice that Lemma~\ref{lem:two-kays} tells us that if a reduced word has more than one $K$, there is a reduced word with at most one letter $K$, so we will restrict our consideration to the latter case.
Further, we will only consider reduced words with as many diagonal generators as possible, as these will encompass all cells with fewer diagonal generators.
\begin{theorem} \label{thm:n-1-cells}
    Let $w_{0,[i,j]}$ denote the longest length word in the set of permutations of $\{i,i+1,\ldots,j\}$ embedded in $S_n$.
    For example, $w_0 = w_{0,[1,n]}$, and $w_{0,[1,n-1]} = (n-1,n-2,\ldots,1,n)$.

    Then all words with at most one $K$ and a maximal number of diagonal generators are equal to one of the following distinct reduced words:
    \begin{enumerate}[label=(\alph*)]
        \item any factorization scheme of type $(\sigma, \omega)$, where $\sigma,\omega \in S_n$.
        \item $W\gamma$, where $W$ is a factorization scheme of $(\sigma,\omega)$ with $\sigma \leq w_{0,[1,n-1]}$, $\omega \leq w_{0,[2,n]}$, using only \circled{2} through \circled{$n$}, and $\gamma \in \Gamma = \{K, (\overline{1})K, (n-1)K, (\overline{1})(n-1)K\}$
    \end{enumerate}
\end{theorem}
\begin{proof}
    It suffices to only consider reduced words, and we can ignore \circled{$i$}'s in this computation, since they do not factor into reduced-ness.
    (We use only $n-1$ generators in the second case, since this will produce the bijection of the parameter map $x_W$, as we will see in \ref{thm:n-1-homeo}.)
    Clearly, when our reduced word contains no $K$, the reduced words are as in (a).

    Now, we consider words with a $K$.
    Without loss of generality, we can move $K$ to the end of the word.
    Further, our relations give us that $f_1$ and $e_{n-1}$ (if present to the left of a $K$) commute with everything.
    For example,
    	\[ (n-1)(n-2)K = K(\overline{n-1})(n-1) = K(n-1)(\overline{n-1}) = (n-2)(n-1)K \]
    As such, a reduced word can only have at most one of each of these; we account for this with our $\gamma$ word, and consider the resulting $e_i$s and $f_i$s, that is, our $W$.

    Notice that when we only consider words without $(n-1)$ or $\overline{1}$, $\gamma$ (or more specifically, $K$) obeys the precisely the same relations as the word
        \[ \alpha = (n-1)(n-2)\cdots(1)(\overline{1})\cdots(\overline{n-2})(\overline{n-1}) \]
    Thus, $W\gamma$ is reduced if and only if $W\alpha$ is reduced (technically, we did not need to restrict to words without $(n-1)$ or $\overline{1}$, since words containing these can never be reduced).\footnote{Here, we suppress the argument that there are no cases where relations can be performed ``inside'' the $\alpha$ that may allow relations there that cannot occur when considering $K$. This should be clear from examining the behavior of moving an element through $\alpha$.}
    We can consider the $e_i$'s and $f_i$'s individually, so we will split $\alpha$ into $a$ and $\overline{a}$ and split $W$ into $\sigma$ and $\omega$.
    To enumerate the distinct $\sigma$ such that $\sigma a$ is reduced, we simply need to find the interval above $a$ in the left weak Bruhat order.
    By Proposition~3.1.6 of \cite{bjorner-brenti}, this occurs precisely when $\sigma$ is less than $w_{0,[1,n-1]}$ in the weak Bruhat order.
    However, since this interval in the weak Bruhat order is isomorphic to the Bruhat order on $S_{n-1}$, by Proposition~3.1.2(iii) of \cite{bjorner-brenti}, this is actually the same as the interval in the strong Bruhat order.
    The analogous statement works for the $f_i$'s.
    Distinctness of the words follows from the distinctness of the corresponding permutations.
\end{proof}

Notice that we relied on the somewhat magical fact that the permutations in our cells were part of a subgroup isomorphic to $S_{n-1}$.
We will see this in the $n-2$ case as well.

Now, we can show that the reduced words described in Theorem~\ref{thm:n-1-cells} reflect the topology of the semigroup.
\begin{theorem}\label{thm:n-1-disjoint}
    For reduced words $A,B$ given by Theorem~\ref{thm:n-1-cells}, if $A \neq B$ then $V(A)$ and $V(B)$ are disjoint.
    As a result, these $V(A)$ partition the semigroup of $(n-1)$-nonnegative invertible matrices.
\end{theorem}
\begin{proof}
    We know all of the cells without $K$'s are disjoint, as it has been established that cells are disjoint in the totally nonnegative case by Theorem~\ref{thm:tnn-decomp}. The cells without $K$'s can be distinguished from the cells with $K$'s by their positive determinants.
    Now, consider a cell containing a $K$.
    We can characterize the cell via $(\sigma,\omega,\gamma)$, where the variables are taken as their respective letters in Theorem~\ref{thm:n-1-cells}.
    Using Lemma~\ref{lem:cell-mult}, and the fact that $\gamma$ is in the double Bruhat cell corresponding to $\alpha$ (that is, $\gamma \in B_\alpha$), we know that we can distinguish cells with different $\sigma$ or $\omega$, since they must lie in different double Bruhat cells.
    To distinguish between the four options of $\gamma$, notice that $(n-1)$ appears in the cell word precisely when the minor indexed by $[1,n-1],[1,n-1]$ is nonzero and $\overline{1}$ appears precisely when the minor indexed by $[2,n],[2,n]$ is nonzero.
    Thus, every matrix in $G$ is in exactly one cell.
\end{proof}

Each of the cells $V(W)$ is homeomorphic to an open ball, as is proved in the following theorem.
We take the standard topology on $GL_n(\mathbb{R})$.
\begin{theorem} \label{thm:n-1-homeo}
    For a reduced word $W$ with at most one $K$ and a maximal number of diagonal generators, $x_W$ is a homeomorphism.
\end{theorem}
\begin{proof}
    First, notice that it is enough to prove the statement for a single representative of each equivalence class, since the relations give homeomorphisms between parameters.
    We will take this choice to be the one given by Theorem~\ref{thm:n-1-cells}, and further, we will choose the representative such that all of the diagonal generators are on the right side of the word.
    We know the result for words without $K$, so we assume that $W$ has a $K$, and so $W = W_1\cdots W_iL$ where $L = K\circled{1}\cdots\circled{$n-1$}$.
    Suppose we have two sets of parameters that map to the same matrix.
    Then
        \[ x_{W_1}(a_1) \cdots x_{W_i}(a_i) x_{L}(a_{i+1}, \ldots, a_j) = x_{W_1}(a_1') \ldots x_{W_i}(a_i') x_{L}(a_{i+1}',\ldots,a_j')\]
    By eliminating equal-valued generators, without loss of generality, we can assume that $a_1 \neq a_1'$, and also that $a_1 > a_1'$.
    If this parameter lies in the $K$ or one of the $h_i$s, then by computation we know that $a_i = a_i'$ for all $i$.
    Otherwise,
        \[ x_{W_1}(a_1-a_1') \cdots x_{W_i}(a_i) x_{L}(a_{i+1}, \ldots, a_j) = x_{W_2}(a_2') \ldots x_{W_i}(a_i') x_{L}(a_{i+1}',\ldots,a_j')\]
    and so this matrix is in two different cells, which gives a contradiction by Theorem~\ref{thm:n-1-disjoint}.

    Thus, the map is a bijection (that the map is surjective is obvious).
    We now only need to show that the map and its inverse are continuous.
    Clearly, the forward map is continuous, since we can express the matrix entries as polynomials in the parameters.

    For the inverse map, first note that $x_L$ is a homeomorphism onto its image, since we can give an explicit rational inverse map.
    We consider the functions that give the parameters of the factorization based on the word $W$ from the matrix entries.
    If $W = W_1\cdots W_kL$, we first determine the parameter $a_1$ of $x_{W_1}$, which we suppose is of type $e$. This must be the maximum value of $a_1$ that will leave the matrix $(n-1)$-nonnegative, since otherwise this would violate Lemma~\ref{bruhat-description}.
    Thus, from Lemma~\ref{lem:tech}, $a_1$ will be the minimum value of the set of $a$'s that make any minor zero.
    Since $a_1$ is the minimum of a number of continuous functions, $a_1$ is itself determined by a continuous function in the entries.
    A similar argument can show that $a_1$ is a continuous function in the entries if $x_{W_1}$ is of type $f$.
    We can then recurse on the resulting matrix to obtain all of our parameters.
\end{proof}

The previous theorem also tells us about factorizations of $(n-1)$-positive matrices, which may be useful in trying to generalize results regarding total positivity, such as the planar networks in \cite{fomin-zel}.
\begin{remark}
    $(n-1)$-positive matrices that are not totally positive, that is, $(n-1)$-positive matrices with a negative determinant, can be factored into the following form:
    	\[ w_{0,[1,n-1]} (n-1) K (1) \overline{w_{0,[1,n-1]}} \circled{1}\circled{2}\cdots\circled{$n-1$} \]
    Here, $w_{0,[1,n-1]}$ signals any reduced word for the long word in $S_{n-1} \hookrightarrow S_n$; for example, in $S_5$, this could be $e_3e_2e_3e_1e_2e_3$.
    The parameter map is a bijection, so the $(n-1)(n-2)+2+(2n-3)+(n-1) = n^2$ parameters are a homeomorphic characterization of matrices of this form.
\end{remark}

As in the TNN case, for a cell $V(W)$, we can consider setting parameters of $e_i$ and $f_i$ to zero (we cannot do this for $K$ without getting a singular matrix), which gives elements in the closure of $V(W)$.
Further, if we can achieve a subword in this way, then the whole cell corresponding to the subword is in the closure.
The question is whether this is everything what we can get in the closure.

Another way to consider this is via the subword order in $\mathbb{S}$.
We can describe the subword order by describing the subwords of each letter, and saying that $V \leq W$ if there is a representation of $W$ (in the equivalence class, that is) such that by replacing every letter in $W$ with a subword of that letter, we can get a representation of $V$.
The typical conception of subword, and the one used in the TNN case, is by defining the subwords of a letter $a$ as $a$ itself and the empty word.
Since we want this to reflect closure, for our alphabet, we will define subwords in this typical sense, except we say that the only subword of $K$ is $K$ itself.

Clearly, this order has the property that $A \leq B \implies V(A) \subset \overline{V(B)}$.
We want to say that $\overline{V(B)}$ contains exactly the cells $V(A)$ such that $A \leq B$.

One statement that will help is noticing that the fact that the closure of cells with negative determinant has no elements with positive determinant.
In fact, we can say slightly more.
\begin{proposition} \label{prop:components}
    The $(n-1)$-nonnegative matrices consist of two path-connected components, those with negative determinant and those with positive determinant.\footnote{Obviously, there are at least two components for the space given by general $k$, since the space is split into matrices with positive and with negative determinant, but we do not have a proof that there are exactly two.}
\end{proposition}
\begin{proof}
    Take some $(n-1)$-nonnegative matrix $M$.
    Consider some minor that is 0.
    If we cannot affect this minor with Chevalley matrices, then it must be the case that that $M$ is not invertible.
    Thus, we can always make minors nonzero.
    This means that by multiplying by Chevalley matrices, we can always force $M$ to be $(n-1)$-positive, and the determinant itself which must stay the same sign.
    We know that totally positive matrices are homeomorphic to a ball, so that component is path-connected.
    From Theorem~\ref{thm:n-1-homeo}, the set of matrices that are $(n-1)$-positive with negative determinant is also homeomorphic to a ball.
    Thus, the negative component is path-connected as well.
\end{proof}

We proceed to prove the original statement.
\begin{theorem}
    The closure $\overline{V(B)}$ is exactly the disjoint union of $V(A)$ for $A \leq B$ in the defined subword order.
\end{theorem}
\begin{proof}
    By \ref{prop:components}, we can consider negative and positive determinant parts of the space separately.
    For the positive component, this is a known result.
    For the negative component, we simply use Lemma~\ref{lem:closure-necc} to see that the only matrices in the cell of $B = (\sigma, \omega, \gamma)$ are those in the double Bruhat cells below $\sigma\omega\alpha$.
    Because elements in the closure of $(n-1)$-nonnegative matrices are $(n-1)$-nonnegative, these can only be the cells $A = (\sigma',\omega',\gamma')$ where $\sigma' \leq \sigma$ and $\omega' \leq \omega$.
    Further, notice that if a particular minor is zero for all elements in the cell, it must remain zero for elements in the closure of the cell.
    Thus, the same argument for disjointness works to show that $\gamma' \leq \gamma$ as well.
    Given these restrictions, all of these can be formed by taking subwords of $B$.
\end{proof}

Since we cannot decompose $K$ into subwords, the poset that naturally results from taking closures is easy to describe: there are two connected components, corresponding to positive and negative determinant.
The positive part is exactly the poset we get from the TNN case.
The negative part is isomorphic to the cartesian product of the interval between the identity and $(w_{0,[1,n-1]},w_{0,[2,n]})$ in the strong Bruhat order with the Boolean algebra on two elements, corresponding to $\{K,(\overline{1})K,(n-1)K,(\overline{1})(n-1)K\}$.
This gives us the following:
\begin{proposition}
    Both parts of the poset are graded via the length function, and have a top and bottom element.
    The positive component is Eulerian, making the poset as a whole trivially semi-Eulerian.
\end{proposition}

\subsection{Cells of \texorpdfstring{$(n-2)$}{n-2}-nonnegative unitriangular matrices}\label{subsec-n-2}

We will perform the same arguments in the $(n-2)$-nonnegative unitriangular case.
Since we are working in the unitriangular setting, we only consider $e_i$ in our factorizations.
Thus, we have a monoid $\mathbb{T} = \langle 1,2,\ldots,n-1,T\rangle$ where our factorizations lie.
As before, we define length by number of parameters (so $\ell(T) = 2n-5$),  and we define equality modulo the relations on $e_i$'s and \ref{n-2-rels}.
We call $W(A)$ the set of matrices corresponding to the factorization $A \in \mathbb{T}$.
Because of Lemma~\ref{lem:two-tees}, we will consider cells with at most one $T$.

We want to say that our resulting cells, defined by, say, $A$ and $B$, are either disjoint or equal (so we can use our relations to move from $A$ to $B$).
However, we need to resolve the issue arising in \ref{n-2-cellsplit}, since this is a case where one cell is the disjoint union of three others.
There is a choice that can be made here. Either we throw out the smaller cells and use the larger cell, or vice versa.
We will call choosing the larger cells the \emph{coarse} choice, and the smaller cells the \emph{fine} choice.
The fine cells distinguish the TNN cells from the $(n-2)$ cells more clearly, while the coarse cells represents more closely the process of Theorem~\ref{thm:decomp}, since we do not distinguish between TNN and non-TNN matrices, and we have a top element.
Despite the difference, none of the proofs to come will depend on the choice used.
Unless explicitly stated, all results will hold for both cases.

We are now ready to enumerate the cells that, as we will later show, partition the set of $(n-2)$-nonnegative unitriangular matrices.

\begin{theorem}\label{thm:n-2-cells}
    Further, let $\alpha = (n-2)\cdots(1)(n-1)\cdots(1)$. Let $\beta = (n-2)\cdots(1)(n-1)\cdots(2)$. Let $w_{0,[n-2]} = (n-2,n-3,\ldots,1,n-1,n)$.

    Then all words with at most one $T$ are equal to one of the following distinct reduced \emph{fine} words:
        \[ w \in \begin{cases}
            w' \lambda & w' \leq w_{0,[n-2]}, \lambda \in \{T, (n-1)T, (n-2)T, (n-2)(n-1)T\} \\
            w'
        \end{cases} \]
    And the following is a complete list of distinct reduced \emph{coarse} words:
        \[ w \in \begin{cases}
            w' \lambda & w' \leq w_{0,[n-2]}, \lambda \in \{T, (n-1)T, (n-2)T, (n-1)(n-2)T\} \\
            w' & w' \ngtr \beta
        \end{cases}\]
        where $w'$ does not involve $T$.
\end{theorem}
\begin{proof}
First, notice that for words with a $T$, $n-1$ and $n-2$ commute with everything except for each other.
Thus, we can group these together into a $\lambda$, and only concern ourselves with $w'\lambda$.

The reduced words of $w'\lambda$ are in bijection with the reduced words of $w'\alpha$.
This is, again, simply a matter of computation (to determine that the relations of $\lambda$ and $\alpha$ are equal for reduced words) and an argument about behavior of relations (namely, that all relations that can possibly be performed with $\alpha$ can also be performed with $\lambda$).
Of particular note is the fact that reduced words cannot contain any $(n-2)$ or $(n-1)$ letters (except in the $\lambda$ or $\alpha$, that is).

Thus, in the same vein as Theorem~\ref{thm:n-1-cells}, using 3.1.6 of \cite{bjorner-brenti}, and the fact that the weak interval and strong interval are equal in this case (the surprising consequence we also see in the $n-1$ case), we get the statement.
\end{proof}

The following statement confirms that our list of relations is complete.
\begin{theorem}\label{thm:n-2-disjoint}
    For reduced words $A$ and $B$ as qualified by Theorem~\ref{thm:n-2-cells}, if $A \neq B$ then $W(A)$ and $W(B)$ are disjoint.
\end{theorem}
\begin{proof}
    It is enough to show this for the fine cells, since no two coarse cells contain the same fine cell.
    If the cell does not have a $T$ this is a known result.
    Thus, suppose $A = a\lambda$ and $B = b\lambda'$, using the notation in Theorem~\ref{thm:n-2-cells}.
    All matrices in $W(\lambda)$ are in the standard Bruhat cell $B_\alpha^-$.
    Thus, Lemma~\ref{lem:cell-mult} tells us if $a \neq b$, then $W(A) \neq W(B)$ because they are in different Bruhat cells. 
    So, all of our cells are distinct, up to containing $n-2$ and $n-1$.
    But we know how to distinguish these: they appear precisely when the minor indexed by $[1,n-1],[1,n-1]$ and the minor indexed by $[2,n],[2,n]$ are nonzero, respectively.
    So a matrix cannot be in $W(A)$ and $W(B)$ unless $\lambda = \lambda'$, and we are done.
\end{proof}

Further, these cells are homeomorphic to open balls.
The proof is essentially the same as the $(n-1)$ case.
\begin{theorem}
    For $A$ a reduced word with at most one $T$, $x_A$ is a homeomorphism.
\end{theorem}

Now, we can consider taking the closure.
This is more complex than the $(n-1)$ case, since the generator $T$ does decompose into other cells.
\begin{lemma} \label{lem:T-params}
    By setting parameters of $T$ to zero, we get matrices that correspond to permutations that are below at least one of the permutations described below in the Bruhat order.
    \begin{enumerate}[label=(\alph*)]
        \item $T_1^i = e_{n-3} \cdots e_1 \;e_{n-1}\cdots \hat{e_{i}}\cdots e_2$, where $2 \leq i \leq n-1$.
        \item $T_2^i = e_{n-2} \cdots \hat{e_{i}} \cdots e_1 \;e_{n-1}\cdots \hat{e_{i+1}}\cdots e_2,$ where  $1 \leq i \leq n-3$.
    \end{enumerate}
    Generators with a cap represent missing generators.
\end{lemma}
\begin{proof}
If one of the $b_i$'s in the parametrization of $T$ is $0$, then it is straightforward to verify by computation that the word
$$e_{n-3}(a_{n-3}) \cdots e_1(a_1) \;e_{n-1}(X) e_{n-2}(b_{n-3})\cdots \hat{e_{i+1}(b_i)}\cdots e_3(b_2) e_2(b_1)$$
describes a factorization for the matrix $T(\vec{a},\vec{b})$. Note that $X$ can be expressed in terms of $\vec{a}$ and $\vec{b}$ and $Y = 0$ in this case. Conversely, every matrix of such a factorization, for $2\leq i \leq n-1$, is in the closure of $T$ and corresponds to the matrix $T(\vec{a},\vec{b})$ where $b_{i-1} = 0$.

If one of the $a_i$'s is $0$, then the factorization is
$$e_{n-2}(B_{n-3}) e_{n-3}(A_{n-3}) \cdots e_{i+1}(A_{i+1}) \hat{e_{i}(a_i)} \cdots e_1(a_1) \;e_{n-1}(X) \cdots e_{i+2}(B_{i+1}) \hat{e_{i+1}(b_{i})}\cdots e_2(b_1),$$
where $A_{i+1} = a_{i+1}+c_i$, $B_{k} = a_kb_k/A_{k}$ for $i+1 \leq k \leq n-3$ and $A_k= a_k+b_{k-1}-B_{k-1}$ for $i+2 \leq k \leq n-3$. $X$ is the usual polynomial calculated from $\vec{a}$ and $\vec{b}$, computed after setting $a_i = 0$, and finally, $B_{n-2}$ is given by $b_{n-3} - B_{n-3} $. Every matrix of such a factorization with $1\leq i \leq n-3$ corresponds to the the matrix where $a_i$ is set to $0$.
\end{proof}
We note that though the $T_1^i$'s and $T_2^i$'s are distinguished here based on whether a $b_i$ or an $a_i$ is $0$, the factorization given by $T^{n-1}_1$ is identical to what would be $T_2^{n-2}$, were it defined, so for ease of notation in our following statements we will sometimes refer to $T_2^{n-2}$.

The statement above inspires us to extend the partial Bruhat ordering on words to another partial ordering on all words, including those with a $T$-generator.
We define the subwords of $T$ to be the reduced words described in Lemma~\ref{lem:T-params}. 
This naturally extends to a general subword order, as described in Section~\ref{subsec-n-1}.
As in the $n-1$ case,  $A \leq B \implies W(A) \subset \overline{W(B)}$ by mapping the closure of the parameter space to the closure of the cell.
Further, every element of the closure of the cell can be achieved by setting parameters to zero, which follows from Lemma~$\ref{lem:T-params}$.

We prove that this subword order exactly describes the closures of cells.
To prove this, we will describe the closure of $\Lambda = \{T,(n-1)T,(n-2)T,(n-2)(n-1)T, (n-1)(n-2)T\}$ in two ways, through subwords and through determinants, and together these will give a straightforward characterization.

Note that throughout the rest of the section we will conflate Bruhat order on permutations with subword order in words in the alphabet $\mathbb{T}$, when the generator $T$ is not present in the words.
Though this is technically an abuse of notation, we will often switch between the two without comment since the two orders are identical in a natural way.
This means that for a word $A \in \mathbb{T}$, if it contains no $T$ letters, we can think of $A$ as both a word and as a permutation.  Thus, we will talk about $W(A)$ as well as $A(i)$.

\begin{proposition} \label{prop:T-closure}
\hfill
\begin{enumerate}
    \item If a matrix $M$ is in the closure of an element $\lambda \in \Lambda$, then the cell $B_m^-$ corresponding to $M$ must satisfy $m \leq \alpha$.
    \item Further, if $\lambda = T$, $(n-1)T$, $(n-2)T$, or $(n-2)(n-1)T$, and $M$ is TNN, then we also require that, respectively, $m(1) \not\in \{n-1,n\}$ and $m(2) \neq n$; $m(1) \not\in \{n-1,n\}$; $m(1),m(2) \neq n$; or $m(1) \neq n$.
    \item The closure of $\lambda \in \Lambda$ contains the TNN cells below $\lambda$ in the subword order.
    \item The TNN cells below $\lambda \in \Lambda$ in the subword order are exactly the ones that satisfy the conditions in (2).
\end{enumerate}
\end{proposition}
Together, this tells us that the subword order reflects the topological structure via the closure order, at least up to elements in $\Lambda$.
\begin{proof} \hfill
\begin{enumerate}
    \item This follows from Lemma~\ref{lem:closure-necc}.

    \item In all of the cells listed, the $[2,n],[1,n-1]$ minor is negative.
    Thus, for a TNN matrix to be in the closure, this minor is zero, unless we are in the cell $(n-1)(n-2)T$, in which case it can be arbitrary, so we do not need the condition.
    By Lemma~\ref{bruhat-description}, the requirements $m(1) \neq n$, $m(1) \neq n-1$, and $m(2) = n$ are equivalent to, respectively, $M_{[2,n],[1,n-1]}$, $M_{[2,n-1],[1,n-2]}$, and $M_{[3,n],[2,n-1]}$ having determinant zero for every $M$ in $V(m)$.
    Depending on the cell, elements in the cell may have these conditions hold true, and it occurs precisely in the manner described above.

    \item This follows immediately from Lemma~\ref{lem:T-params}.

    \item This argument uses the fact that the elements below $m$ in the Bruhat order are precisely the subwords of some particular reduced expression for $m$ (the subword property).
    This will consist of mostly computations, but we will sketch the important parts, and suppress the actual calculations.

    First, we can see from Lemma~\ref{lem:bb-bruhat} that if $\sigma$ satisfies the set of conditions corresponding to any particular $\lambda \in \Lambda$, as specified in (2), then so does everything below $\sigma$.
    We will use this fact throughout.
    We begin with $(n-2)(n-1)T$, which has the fewest conditions.

    Recall that $\alpha$ was defined as $\alpha = (n-2)\cdots(1)(n-1)\cdots(1)$.
    Notice that $\alpha$ sends 1 to $n$ (it is easy to see this by considering $(i)$ as the function swapping $i$ with $i+1$).
    Now, consider some subword $m \leq \alpha$.
    In order for $m$ to satisfy $m(1) \neq n$, we must remove at least one letter from the $(n-1)\cdots(1)$ portion of the word.
    If we remove the $(1)$, then the result is a reduced word.
    If we remove something else, such as $(i)$, then the reduced expression for $m$ is below the expression we get when we just remove $(1)$, as can be seen by computation.
    Thus, the $m$ must be below
        \[ (n-2)\cdots(1)(n-1)\cdots(2) \]
    All $m$ that satisfy this condition also satisfy the condition that $m(1) \neq n$, so we have shown that these exactly define our desired cells.
    This can be formed from a subword of $(n-2)(n-1)\alpha$:
        \[ (n-2)\cdots(1)(n-1)\cdots(2) = (n-2)(n-1)T_1^{n-2} \]
    Now, consider adding the condition that $m(1) \neq n-1$, as is the case for $(n-1)T$.
    $\alpha$ does not satisfy this condition, so we must consider proper subwords.
    Obviously, we must remove a letter from the first descending sequence.
    From computation it turns out that this is all we need, and we get that any reduced subword that satisfies our conditions is below one of
        \[ (n-2)\cdots(\hat{i})\cdots(1)(n-1)\cdots(2) = (n-1)T_2^i \qquad 1 \leq i \leq n-2 \]
    Similarly, if we have the condition that $m(1),m(2) \neq n$, as in the $(n-2)T$ case, we simply need to remove a letter from the second descending sequence, and reduced subwords must be below one of
        \[ (n-2)\cdots(1)(n-1)\cdots(\hat{i})\cdots(2) = (n-2)T_1^i \qquad 1 \leq i \leq n-2 \]
    Finally, for the $T$ case, if we add both conditions, we can consider the words above, except removing one of the letters in the first descending sequence. This gives precisely the words below $T$.
    We can see this from rewriting the above as
        \[ (n-2)\cdots(i)(n-1)\cdots(i+1)(i-1)\cdots(1)(i-1)\cdots(2) \]
    We must remove a letter in the first or third descending sequence of the above word; if we remove one from the first sequence, then the first half is a subword of the word achieved from removing $(n-2)$ (as in $T_1$).
    If we remove a letter from the third sequence, then the second half is a subword of the word achieved from removing $(i-1)$ (as in $T_2$).

    Finally, $\alpha = (n-1)(n-2)T_2^{n-2}$.  The cells below $\alpha$ are exactly those that have been described above.
\end{enumerate}
\end{proof}
Now, we will consider the closures of our cells.
Note that any element in the closure of a cell must be $(n-2)$-nonnegative and unitriangular, so we know the closure must be contained in the disjoint union of some set of cells.
We will show that the cells of subwords are enough, and because we know these are contained in the closure, we will get that the closure is precisely this union of these cells.
That is, the poset given by subword order is equal to the poset given by closure order (defined as $A \leq B \iff W(A) \subset \overline{W(B)}$).

\begin{theorem}
    The closure $\overline{W(B)}$ exactly consists of all $W(A)$ for all $A\leq B$ in our defined subword order.
\end{theorem}

\begin{proof}
    We know the backwards direction; that is, if $A \leq B$ in the subword order, then $W(A) \subset \overline{W(B)}$.
    This uses the fact that the $x_B$ maps the closure of the parameter space into the closure of the image, and so we can set parameters to zero (in other words, take subwords) to get full cells in the closure.

    So, it suffices to show that if $W(A)$ intersects $\overline{W(B)}$, then $A \leq B$ in the subword order.
    If both $A$ and $B$ are TNN, then this is a known result.
    The situation that $A$ is not TNN but $B$ is TNN cannot occur, since an element in the closure of $B$ must be TNN.

    Now, suppose $B$ is not TNN, so $B$ can be written as $b\lambda$ a reduced word where $b$ does not include $(n-1)$ or $(n-2)$ and $\lambda \in \Lambda$.
    When $A$ is also not TNN, and is split into $a\lambda'$ in a similar way, Lemma~\ref{lem:cell-mult} and properties of the standard Bruhat decomposition require that $a\alpha \leq b\alpha$.
    The non-TNN cells in standard Bruhat cells $B_p^-$ for $p \leq b\alpha$ are precisely $W(s_1s_2)$, where $s_1 \leq b$ and $s_2 \in \Lambda$.
    
    By looking at the minors for different values of $s_2$, it is easy to see that we must have $\lambda' \leq \lambda$.
    For example, $\overline{(n-1)T}$ cannot intersect $(n-1)(n-2)T$, since elements of $(n-1)T$, and thus its closure, must have the top-left large minor be zero, which is always positive in $(n-1)(n-2)T$.
    The same argument with the minor description given by Lemma~\ref{bruhat-description} also works to show that we must have $a \leq b$.
    
    From here, the possible cells that could intersect the closure of $W(B)$ are precisely $W(A)$ for $A$ a subword of $B$.

    When $A$ is TNN, then notice that the restrictions in Lemma~\ref{prop:T-closure}(2) corresponding to $\lambda$ must also apply to $A$.
    Thus, we must take a subword of $\alpha$ where these identities hold.
    By Lemma~\ref{lem:closure-necc}, these are exactly the subwords of $T$.
\end{proof}

All of the above work gives us enough structure to say the following:
\begin{corollary}
    Both the coarse and the fine cells form a CW-complex of the space of $(n-2)$-nonnegative unitriangular matrices.
\end{corollary}

Specifically, the closure poset corresponding to the CW-complex is given by the subword order.

To conclude, we will prove that the poset on cells $W(A)$ is a graded poset.
The choice of coarse or fine cells does not matter here, since the proof is based on the fact that the vast majority of the poset is lifted from the standard Bruhat order.

First, we generalize the exchange condition of $S_n$ to our generating set, where we have the relation $e_ie_i = e_i$, as opposed to the Coxeter relation $s_is_i=1$.
\begin{lemma}
    Let $w$ be a word in $\Lambda = \{1,\ldots,n-1\}$ subject to Chevalley relations, that is, the following relations hold:
    \begin{itemize}
        \item $i\;i \leftrightarrow i$ (the shortening relation)
        \item $i\;j\;i \leftrightarrow j\;i\;j$ if $|i - j| = 1$ (the adjacent relation)
        \item $i\;j \leftrightarrow j\;i$ if $|i - j| > 1$ (the nonadjacent relation)
    \end{itemize}
    Then if $w$ is a reduced word, for $t \in \Lambda$, exactly one of the following is true:
    \begin{itemize}
        \item $tw$ is reduced, so $\ell(tw) = \ell(w) + 1$;
        \item $tw = w$, so $\ell(tw) = \ell(w)$.
    \end{itemize}
\end{lemma}
\begin{proof}
    Suppose $tw$ is not reduced. Let $M = (m_1,m_2,\ldots,m_r)$ be a sequence with $tw = m_1$, each $m_k$ at most one local move away from the previous, with no $i \to i\;i$ moves, and $m_r$ a reduced word.

    To see that there always exists such a sequence, consider the following. We know such a sequence exists for Coxeter groups from \cite{bjorner-brenti} (cf. Theorem 3.3.1). Mimic this sequence until we hit an $s_is_i=1$ relation. Here, use the shortening relation instead, and continue with the resulting smaller word.

    We will use a function $\varphi$ to indicate a sort of location for $t$ as we move along the sequence $M$.  Define $\varphi: M \to [\ell(w)+1] \cup \varnothing$ recursively in the following way:
    \begin{align*}
        \varphi(m_1) &= 1 \\
        \varphi(m_i) &= \begin{cases}
            \varnothing & \text{$t$ is in shortening relation from $m_{i-1}$ to $m_i$ or } \varphi(m_{i-1}) = \varnothing \\
            \varphi(m_{i-1}) \pm 1 & \text{$t$ is in left/right position of nonadjacent relation used from $m_{i-1}$ to $m_i$} \\
            \varphi(m_{i-1})\pm 2 & \text{$t$ is in left/right position of adjacent relation used from $m_{i-1}$ to $m_i$} \\
            \varphi(m_{i-1})-1 & \text{$t$ to the right of shortening relation used from $m_{i-1}$ to $m_i$} \\
            \varphi(m_{i-1}) & \text{otherwise}
        \end{cases}
    \end{align*}
    \begin{lemma}
        The following properties are true:
        \begin{enumerate}[label=(\alph*)]
            \item $\varphi$ is well-defined.
            \item There are no length-shortening moves that don't involve $t$.
            \item $m_r$ is a reduced word for $w$.
        \end{enumerate}
    \end{lemma}
    \begin{proof} \hfill
        \begin{enumerate}[label=(\alph*)]
            \item Because we chose the sequence such that we never get a longer word than $tw$, our codomain is correctly stated. Thus, the only thing to check for well-definedness is whether $\varphi(m_{i-1})$ can ever be in the middle of an adjacent relation.

            Let $n_i$ be $m_i$ with the $\varphi(m_i)^{\text{th}}$ letter in the word removed, where we take out nothing if $\varphi(m_i) = \varnothing$. That is, we take out the $t$ from the word, and $\varnothing$ signifies that the $t$ no longer exists. Then notice that $w = n_1$, and each $n_i$ is at most one local move away from $n_{i-1}$. The reason for this is that removing $t$ does not affect any local moves not involving $t$, and the local moves that do involve $t$ don't affect anything except $t$. Suppose we do have a move where the location of the $t$ is in the center. Then if we consider the $n_i$ up to that point, we get that there is an $n_i$ with two adjacent identical letters:
            \[ m_i = \cdots j\; t\; j \cdots \implies n_i = \cdots j\;j \cdots \]
            However, this would imply that $w$ can be reduced to a word of shorter length. This is a contradiction.

            \item The same reasoning applies. Consider the $n_i$. If there was a length-shortening move then obviously we would get that $n_i$ is a series of moves that shortens $w$, which is not possible.

            \item We must have a shortening relation to get a reduced word. This relation must contain $t$, so there must be exactly one. Notice that once $\varphi(m_i) = \varnothing$, $m_i \equiv n_i$ modulo the Chevalley equivalence relations. We know that $w \equiv n_i$ for all $i$. Thus, $m_i \equiv w$.
        \end{enumerate}
    \end{proof}
    This gives us the statement.
\end{proof}

\begin{theorem}
    The Bruhat poset is graded.
\end{theorem}
\begin{proof}
    For anything not containing a $T$, this is well-known, one proof being Theorem 2.2.6 of \cite{bjorner-brenti}.
    We know that when we only consider words containing $T$, the restriction gives a poset that is isomorphic to the product poset of an interval in the strong Bruhat poset with the Boolean algebra on 2 elements (corresponding to containment of $(n-1)$ and $(n-2)$).
    Both of these are graded, and it is easy to see that the rank function is equivalent to the sum of the rank functions of the individual posets.
    Thus, the Bruhat poset being graded implies that restricting to this case gives a graded poset.

    By inspection of the subwords of $T$, the interval up to $T$ is also a graded poset. Now, all that is left is to consider working between words with $T$ and without $T$.

    Now, suppose that $wT$ is reduced but reducing $T$ to some subword $t$ makes $wt$ not reduced. We want to show that there is a chain between $wT$ and $wt$ that behaves correctly with respect to our rank function.

    Let $w = w_1\cdots w_a$. Then consider $w_i\cdots w_a t$, starting from $i = a$ to $i = 1$. Using the exchange property, we can see that this reduces to some $w' t$, where $w'$ is a subword of $t$. Thus, this has the intermediary $w' T$, and from the lemma we get intermediaries as desired.
\end{proof}
In addition to being ranked, the poset of the CW-complex of the semigroup of unitriangular totally nonnegative matrices is Eulerian. However, the closure poset of our CW-complex is not Eulerian in general.
\begin{remark}
    The closure poset on cells of $(n-2)$-nonnegative unitriangular matrices is Eulerian for $n \leq 4$.
    For $n > 4$, the poset is not Eulerian: by computation using Lemma \ref{lem:bb-bruhat}, the interval $[(n-2)\cdots(3)(n-1)\cdots(3),T]$ has only three elements, the middle one being $T_2^1$.
\end{remark}

As a result, some questions arise: is the closure poset still shellable, or perhaps semi-Eulerian?
What is the structure of our spaces, since they are not triangulations of a sphere?

As another thing to consider, recall that we can make the choice of whether to use coarse or fine cells in the $n-2$ case.
If we choose the coarse option at a certain level, it must be consistent with all the options above.
Thus, we can consider taking refinements between coarse and fine.
This is a topic for future research, and we have no current results about these.

Finally, it seems reasonable that this poset and cell structure extends to general unitriangular $k$-nonnegative matrices.\footnote{Relatedly, we could consider using our results in the unitriangular case to extend the results to the general $(n-2)$-nonnegative case, but the two unitriangular semigroups, along with diagonal matrices, do not generate the whole space alone, and so it is unclear whether this approach could lead to a natural description of the semigroup.}
This may motivate a more thorough investigation of the general case, where elementary arguments about vanishing minors and $k$-irreducibility begin to fail.

\section*{Acknowledgements}

This research was carried out as part of the 2017 REU program at the School of Mathematics at University of Minnesota, Twin Cities.
Sunita Chepuri, Joe Suk, and Ewin Tang are grateful for the support of NSF RTG grant DMS-1148634 and Neeraja Kulkarni for the support of NSF grant DMS-1351590.
The authors would like to thank Pavlo Pylyavskyy and Victor Reiner for their guidance and encouragement.
They are also grateful to Elizabeth Kelley for her many helpful comments on the manuscript and to Anna Brosowsky and Alex Mason for their help and support working on $k$-nonnegativity and writing a joint manuscript.

\appendix

\section{Minors in Full} \label{app:minors}

\begin{lemma*} \label{lem:minors-list}
    The solid minors of $K(\vec{a},\vec{b})$ are as follows. First, the minors on the subdiagonal and superdiagonal:
    \begin{align*}
        \left|K(\vec{a},\vec{b})_{[i,j],[i+1,j+1]}\right| &= \prod_{k=i}^j a_kb_k \\
        \left|K(\vec{a},\vec{b})_{[i,n-1],[i+1,n]}\right| &= b_1\cdots b_{n-1}\prod_{k=i}^{n-1} a_kb_k \\
        \left|K(\vec{a},\vec{b})_{[i+1,j+1],[i,j]}\right| &= 1
    \end{align*}
    Then, the principal minors:
    \begin{align*}
        \left|K(\vec{a},\vec{b})_{[i,j],[i,j]}\right| &= \sum_{k=i-1}^j\left(\prod_{\ell=i}^k b_{\ell-1} \prod_{\ell=k+1}^j a_\ell\right) && i,j < n,\, \text{where } a_{n-1},b_0 = 0 \\
        \left|K(\vec{a},\vec{b})_{[i,n],[i,n]}\right| &= \left(\prod_{k=i}^n b_{k-1}\prod_{k=i}^{n-2} a_k \right)\left|K(\vec{a},\vec{b})_{[2,i-1],[2,i-1]}\right| && i > 2 \\
        &= \left(\prod_{k=i}^n b_{k-1}\prod_{k=i}^{n-2} a_k \right)\sum_{k=1}^{i-1}\left(\prod_{\ell=2}^k b_{\ell-1} \prod_{\ell=k+1}^{i-1} a_\ell\right) \\
        \left|K(\vec{a},\vec{b})_{[2,n],[2,n]}\right| &= 0 \\
        \left|K(\vec{a},\vec{b})\right| & = -a_1\cdots a_{n-2}b_1\cdots b_{n-1}
    \end{align*}
    All other minors are trivially zero.
\end{lemma*}

\section{Relations in Full} \label{app:rels}

First, the $k = n-1$ relations (that is, for $K(\vec{a},\vec{b})$):

In these relations, the variables on the right-hand side are expressed in terms of the variables on the left-hand side.
\begin{enumerate}[label=(\arabic*)]
\item $e_i(x) K(\vec{a},\vec{b}) = K(\vec{A},\vec{B}) e_{i+1}(x'), \text{ where } 1 \leq i \leq n-2$:

The following equalities hold for $i < n-2$.
\begin{align*}
    \vec{A} &= \left(a_1,\ldots,a_{i-1},a_i + x, \frac{a_ia_{i+1}}{a_i + x}, a_{i+2},\ldots,a_{n-2}\right) \\
    \vec{B} &= \left(b_1,\ldots,b_{i-1},b_i+\frac{xa_{i+1}}{a_i+x}, \frac{b_ib_{i+1}(a_i+x)}{b_i(a_i+x)+xa_{i+1}},b_{i+2},\ldots,b_{n-1}\right) \\
    x' &= \frac{b_{i+1} a_{i+1} x}{b_i(a_i+x)+xa_{i+1}}.\\
    \intertext{and in the other direction,}
     \vec{a} &= \left(A_1,\ldots,A_{i-1},\frac{A_iA_{i+1}B_{i+1}+A_iA_{i+1}x'}{A_{i+1}B_{i+1}+A_{i+1}x'+B_ix'}, A_{i+1}+\frac{B_ix'}{B_{i+1}+x'}, A_{i+2},\ldots,A_{n-3}\right) \\
    \vec{b} &= \left(B_1,\ldots,B_{i-1},\frac{B_iB_{i+1}}{B_{i+1}+x'}, B_{i+1}+x', B_{i+2},\ldots,B_{n-2}\right) \\
    x &= \frac{x'A_iB_i}{A_{i+1}B_{i+1}+A_{i+1}x'+B_ix'}
\intertext{and when $i = n-2$, we have}
    \vec{A} &= \left(a_1,\ldots,a_{n-3},a_{n-2} + x\right) \\
    \vec{B} &= \left(b_1,\ldots,b_{n-2},\frac{b_{n-1}a_{n-2}}{a_{n-2}+x}\right) \\
    x' &= \frac{b_{n-1}\cdots b_2b_1x}{b_{n-2}(a_{n-2}+x)} \\
      \intertext{and in the other direction,}
         \vec{a} &= \left(A_1,\ldots,A_{n-3},\frac{A_{n-2}B_1\cdots B_{n-1}}{B_1\cdots B_{n-1}+x'B_{n-2}}\right) \\
    \vec{b} &= \left(B_1,\ldots,B_{n-2},B_{n-1}+\frac{x'}{B_1\cdots B_{n-3}}\right) \\
    x &= \frac{A_{n-2}B_{n-2}x'}{B_1\cdots B_{n-1}+x'B_{n-2}}
\end{align*}

\item $e_{n-1}(x) K(\vec{a},\vec{b}) = K(\vec{A},\vec{B}) f_{n-1}(x') h_{n-1}(c)$:
\begin{align*}
    c &= \frac{Y}{Y + xX} = \frac{1}{1 + x'X} \\
    \vec{A} &= \vec{a} \\
    \vec{B} &= \left(b_1,\ldots,b_{n-3},\frac{b_{n-2}}{c}, b_{n-1}\right) \\
    x' &= \frac{x}{Y}
\end{align*}

\item $ f_{i+1}(x) K(\vec{a},\vec{b}) = h_{i+2}(1/w) K(\vec{A},\vec{B}) f_i (x) h_{i}(w), \text{ where } 1 \leq i \leq n-2$:
\begin{align*}
\intertext{when $1 \leq i < n-2$, we have:}
    w &= \frac{1}{1+xa_{i+1}+xb_i} \\
    \vec{A} &= \left(a_1,\ldots,a_{i-2},a_{i-1},a_i(xa_{i+1}+1),\frac{a_{i+1}(xa_{i+1}+xb_i+1)}{1+xa_{i+1}}, \frac{a_{i+2}}{xa_{i+1}+xb_{i+1}+1},a_{i+3},\ldots,a_{n-2}\right) \\
    \vec{B} &= \left(b_1,\ldots,b_{i-2},b_{i-1}(xa_{i+1}+xb_i+1), \frac{b_i}{xa_{i+1}+1},\frac{b_{i+1}(1+xa_{i+1})}{xa_{i+1}+xb_i+1}, b_{i+2},\ldots,b_{n-1}\right) \\
    \intertext{and for the other direction:}
      w &= \frac{1}{1+xA_{i+1}+xB_i} \\
     \vec{a} &= \left(A_1,\ldots,A_{i-2},A_{i-1},\frac{A_i(1+xB_i)}{xA_{i+1}+xB_{i}+1},\frac{A_{i+1}}{1+xB_i}, A_{i+2}(xA_{i+1}+xB_{i}+1),A_{i+3},\ldots,A_{n-2}\right) \\
    \vec{b} &= \left(B_1,\ldots,B_{i-2},\frac{B_{i-1}}{xA_{i+1}+xB_i+1}, \frac{B_i(xA_{i+1}+xB_i+1)}{1+xB_i},B_{i+1}(1+xB_i), B_{i+2},\ldots,B_{n-1}\right)
    \intertext{and when $i = n-2$:}
    w &= \frac{1}{1+xb_{n-2}} \text{ and } \vec{A} = \vec{a} \\
    \vec{B} &= \left(b_1,\ldots,b_{n-4},b_{n-3}(xb_{n-2}+1), b_{n-2},\frac{b_{n-1}}{xb_{n-2}+1} \right)
\end{align*}

\item $f_{1}(x) K(\vec{a},\vec{b}) = K(\vec{A},\vec{B}) e_{1}(x') h_1(c)$:
\begin{align*}
    \vec{A} &= \left(\frac{a_1}{1+xa_1},a_2,\ldots,a_{n-2}\right) \\
    \vec{B} &= \vec{b} \\
    x' &= xb_1a_1 \\
    c &= \frac{1}{1+xa_1}
\end{align*}
For the other direction, we use $a_1 = A_1+\frac{A_1x'}{B_1}$ and $c = \frac{B_1}{B_1+x'}$.
\item
    $h_i(x) K(\vec{a}, \vec{b}) = K(\vec{A}, \vec{B}) h_{i-1}(x), \text{ where } 2 \leq i \leq n :$
    \begin{align*}
    \vec{A} &= \left(a_1,\ldots,a_{i-1},xa_i,\frac{a_{i+1}}{x}, a_{i+2},\ldots,a_{n-3}\right) \\
    \vec{B} &= \left( b_1,\ldots,xb_{i-1}, \frac{b_i}{x}, b_{i+1}, b_{i+2},\ldots,b_{n-2}\right)
\end{align*}
\item
    $h_1(x) K(\vec{a}, \vec{b}) = K(\vec{A}, \vec{B}),$ where
    $\vec{A} = \left(xa_1, a_2, \ldots ,a_{n-3}\right)$
    and $\vec{B} = \vec{b}$.
\item
    $K(\vec{a},\vec{b})h_n(x) = K(\vec{A},\vec{B})$, where
    $\vec{A} = \vec{a}$ and
    $\vec{B} = \left(b_1,\ldots,b_{n-3},xb_{n-2}\right)$.
\end{enumerate}

Now, we list the $k = n-2$ relations (that is, for the $T(\vec{a},\vec{b})$ parameter family):
\begin{enumerate}[label=(\arabic*)]
\item
$    e_i(x) T(\vec{a},\vec{b}) = T(\vec{A},\vec{B}) e_{i+2}(x'), \text{ where } 1 \leq i \leq n-3 .$

\begin{align*}
    \vec{A} &= \left(a_1,\ldots,a_{i-1},a_i+x,\frac{a_ia_{i+1}}{a_i+x}, a_{i+2},\ldots,a_{n-3}\right) \\
    \vec{B} &= \left(b_1,\ldots,b_{i-1},b_i+\frac{xa_{i+1}}{x+a_i}, \frac{b_ib_{i+1}(x+a_i)}{b_i(a_i+x)+xa_{i+1}}, b_{i+2},\ldots,b_{n-2}\right) \\
    x' &= \frac{b_{i+1}a_{i+1}x}{b_i(a_i+x)+xa_{i+1}}
\end{align*}
In the other direction, we have:
\begin{align*}
    \vec{a} &= \left(A_1,\ldots,A_{i-1},\frac{A_iA_{i+1}B_{i+1}+A_iA_{i+1}x'}{A_{i+1}B_{i+1}+A_{i+1}x'+B_ix'}, A_{i+1}+\frac{B_ix'}{B_{i+1}+x'}, A_{i+2},\ldots,A_{n-3}\right) \\
    \vec{b} &= \left(B_1,\ldots,B_{i-1},\frac{B_iB_{i+1}}{B_{i+1}+x'}, B_{i+1}+x', B_{i+2},\ldots,B_{n-2}\right) \\
    x &= \frac{x'A_iB_i}{A_{i+1}B_{i+1}+A_{i+1}x'+B_ix'}
\end{align*}
\item
    $e_{n-2}(x) T(\vec{a},\vec{b}) = T(\vec{A},\vec{B}) e_1(x').$

Here $\vec{A}$ and $\vec{B}$ satisfy the following recurrence:
\begin{align*}
    B_{n-3} &= b_{n-3}+x \\
    A_{i} &= (a_i\cdot b_i)/B_i, \text{ where } 1 \leq i \leq n-3 \\
    B_i &= a_{i+1} + b_i - A_{i+1}, \text{ where } 1 \leq i \leq n-4 \\
    x' &= a_1 - A_1.
\end{align*}
(Note that $B_{n-3} > b_{n-3}$, and consequently $A_{n-3} < a_{n-3}$. In turn, $B_{n-2} > b_{n-2}$, etc, so that in general $B_i > b_i$ and $A_i < a_i$.)
In the other direction,
\begin{align*}
    a_1 &= x' + A_1 \\
    c_i &= A_iC_i/a_i\text{ where } 1 \leq i \leq n-3 \\
    a_i &= A_i + C_{i-1} - c_{i-1}, \text{ where } 1 \leq i \leq n-4 \\
    x &= C_{n-3} - c_{n-3}.
\end{align*}
(Similarly, $a_1 > A_1$, consequently $c_2 < C_2$. In turn, $a_2 > A_2$, etc, so that in general $a_i > A_i$ and $c_i < C_i$.)

\item $e_{n-1}(x) T(\vec{a},\vec{b}) = T(\vec{A},\vec{B}) e_2(x')$
\begin{align*}
    \vec{A} &= \vec{a} \\
    \vec{B} &= \left(b_1,\ldots,b_{n-3},b_{n-2}+\frac{b_{n-2}}{b_1x},\right) \\
    x' &= \frac{x}{\left|T(\vec{a},\vec{b})_{[3,n-3],[4,n-2]}\right|}
\end{align*}
In the other direction,
\begin{align*}
    \vec{a} &= \vec{A} \\
    \vec{b} &= \left(B_1,\ldots,B_{n-3},\frac{B_{n-2}B_1}{B_1+x'},\right) \\
    x &= x' \left|T(\vec{A},\vec{B})_{[3,n-3],[4,n-2]}\right|
\end{align*}

\item
$e_{n-1}e_{n-2} T = e_{n-2}e_{n-1}T \sqcup e_{n-2}\cdots e_1 e_{n-1}\cdots e_2 \sqcup e_{n-2}\cdots e_1 e_{n-1}\cdots e_1$.

The three factorizations on the right hand side of the equation arise from three possible values of the minor $\left| M_{[2,n][1,n-1]}\right|$, where $M = e_{n-1}(u)e_{n-2}(v) T(\vec{a},\vec{b}).$
\begin{enumerate}
\item When the minor is negative, then we have:
\begin{align*}
(a_1\ldots a_{n-3})\cdot v \cdot (X+ u) &< (a_1\ldots a_{n-3})\cdot (b_{n-2}Y+ vX) \\
\Rightarrow \; v \cdot (X+ u) &< (b_{n-2}Y+ vX)  \\
\Rightarrow \; vu  &< b_{n-2}Y
\end{align*}
Then the matrix $M$ can be factored as follows.

$$e_{n-1}(u) e_{n-2}(v) T(\vec{a},\vec{b}) = e_{n-2}(v)e_{n-1}(u)T(\vec{A},\vec{B}),$$
where $\vec{A} = \vec{a}$
and $\vec{B} = \left(b_1,\ldots,b_{n-3},b_{n-2} - uv/Y \right)$.

\item When the minor is zero, then we have:
\begin{align*}
(a_1\ldots a_{n-3})\cdot v \cdot (X+ u) &= (a_1\ldots a_{n-3})\cdot (b_{n-2}Y+ vX)  \\
\Rightarrow \; v \cdot (X+u) &= b_{n-2}Y+ vX
\end{align*}
Then the matrix is totally nonnegative and can be factored as shown below.
$$e_{n-1}(u) e_{n-2}(v) T(\vec{a},\vec{b}) = e_{n-2}(v) e_{n-3}(a_{n-3})\cdots e_1(a_1) e_{n-1}(X+u) e_{n-2}(b_{n-3})\cdots e_2(b_1) $$
\item When the minor is positive, then we have:
\begin{align*}
 (a_1\ldots a_{n-3})\cdot v \cdot (X+ u) &> (a_1\ldots a_{n-3})\cdot (b_{n-2}Y+ vX)  \\
\Rightarrow \; v \cdot (X+u) &> b_{n-2}Y+ vX
\end{align*}
Again, the matrix is totally nonnegative and can be factored as written below.
$$e_{n-1}(u) e_{n-2}(v) T(\vec{a},\vec{b}) =  e_{n-2}(v') e_{n-3}(A_{n-3})\cdots e_1(A_1) e_{n-1}(X+u) e_{n-2}(B_{n-3})\cdots e_1(B_0),$$
where $\vec{A}$, $\vec{B}$ and $v'$ can be determined from $\vec{a}$, $\vec{b}$, $u$ and $v$, by recursive formulas:
\begin{align*}
    v' &= \frac{b_{n-2}Y+ vX}{X+u} \\
    B_{n-3} &= b_{n-3} + v - v' \\
    A_{i} &= (a_i\cdot b_i)/B_i, \text{ where } 1 \leq i \leq n-3 \\
    B_i &= a_{i+1} + b_i - A_{i+1}, \text{ where } 0 \leq i \leq n-4
\end{align*}
Note that our calculations above show that $v>v'$, and this will show that $a_i>A_i$ and $b_i < B_i$ for all $i$.
\end{enumerate}
\end{enumerate}


\newpage
\begin{thebibliography}{10}

\bibitem{Berenstein-Fomin-Zelevinsky}
Arkady Berenstein, Sergey Fomin, and Andrei Zelevinsky, \emph{Parametrizations
  of canonical bases and totally positive matrices}, Advances in Mathematics
  \textbf{122} (1996), no.~1, 49 -- 149.

\bibitem{bjorner-brenti}
A.~Bj\"{o}rner and F.~Brenti, \emph{Combinatorics of coxeter groups}, Springer,
  2000.

\bibitem{REUreport}
Anna Brosowsky, Neeraja Kulkarni, Alex Mason, Joe Suk, and Ewin Tang,
  \emph{Parametrizations of k-nonnegative matrices}, U. Minnesota REU Report
  (2017).

\bibitem{fallat-johnson-sokal}
Shaun Fallat, Charles~R. Johnson, and Alan~D. Sokal, \emph{Total positivity of
  sums, hadamard products and hadamard powers: Results and counterexamples},
  Linear Algebra and its Applications \textbf{520} (2017), 242 -- 259.

\bibitem{fallat-johnson}
Shaun~M. Fallat and Charles~R. Johnson, \emph{Hadamard powers and totally
  positive matrices}, Linear Algebra and its Applications \textbf{423} (2007),
  no.~2, 420 -- 427.

\bibitem{TNNbook}
\bysame, \emph{Totally nonnegative matrices}, Princeton University Press, 2011.

\bibitem{fomin-zel-bruhat}
Sergey Fomin and Andrei Zelevinsky, \emph{Double bruhat cells and total
  positivity}, Journal of the American Mathematical Society \textbf{12} (1999),
  no.~2, 335--380.

\bibitem{fomin-zel}
\bysame, \emph{Total positivity: Tests and parametrizations}, The Mathematical
  Intelligencer \textbf{22} (2000), no.~1, 23--33.


\bibitem{Hersh2014}
Patricia Hersh, \emph{Regular cell complexes in total positivity}, Inventiones
  mathematicae \textbf{197} (2014), no.~1, 57--114.

\end{thebibliography}
\end{document}